\documentclass[reqno,10pt]{amsart}
\usepackage{amssymb,amsmath,amsthm}
\usepackage{amscd}
\usepackage{amsfonts}
\usepackage{amssymb}
\usepackage{latexsym}
\usepackage{color}
\usepackage{esint}

\usepackage{graphicx}
\graphicspath{ {images/} }

\usepackage[makeroom]{cancel}

\newtheorem{theorem}{Theorem}

\newtheorem{definition}{Definition}
\newtheorem{lemma}{Lemma}
\newtheorem{proposition}[theorem]{Proposition}
\newtheorem{remark}{Remark}
 \newtheorem*{theorem*}{Rough version of the Main theorem}

\let\p=\partial

\let\O=\Omega

\numberwithin{equation}{section}

\let\hide\iffalse
\let\unhide\fi

\DeclareMathAlphabet{\mathpzc}{OT1}{pzc}{m}{it}

\newcommand{\R}{\mathbb{R}}

\newcommand{\be}{\begin{equation}}
\newcommand{\bm}{\begin{multline}}
\newcommand{\ee}{\end{equation}}
\newcommand{\dd}{\mathrm{d}}

\newcommand{\xb}{x_{\mathbf{b}}}
\newcommand{\tb}{t_{\mathbf{b}}}

\newcommand{\xf}{x_{\mathbf{f}}}
\newcommand{\tf}{t_{\mathbf{f}}}

\newcommand{\Bes}{\begin{eqnarray*}}
\newcommand{\Ees}{\end{eqnarray*}}
\newcommand{\Be}{\begin{equation} }
\newcommand{\Ee}{\end{equation}}
\newcommand{\Bs}{\begin{split}}   

\newcommand{\vertiii}[1]{{\left\vert\kern-0.25ex\left\vert\kern-0.25ex\left\vert #1 
    \right\vert\kern-0.25ex\right\vert\kern-0.25ex\right\vert}}

 \makeatletter
\def\munderbar#1{\underline{\sbox\tw@{$#1$}\dp\tw@\z@\box\tw@}}
\makeatother

\def\p{\partial}

\def\O{\Omega}
\def\R{\mathbb{R}}

\def\B{\begin{equation}}
\def\E{\end{equation}}
\def\BN{\begin{eqnarray*}}
\def\EN{\end{eqnarray*}}

\def\bcb{\begin{color}{blue}}
\def\ec{\end{color}}

\def\bcr{\begin{color}{red}}
\def\ec{\end{color}}

\begin{document}
 % \date{ \today}

\title{Damping of kinetic transport equation with diffuse boundary condition}

 \author{Jiaxin Jin}
 \address{Department of Mathematics, University of Wisconsin-Madison, Madison, WI, 53706, USA, email: jjin43@wisc.edu}
 \author{Chanwoo Kim}
 \address{Department of Mathematics, University of Wisconsin-Madison, Madison, WI, 53706, USA, email: ckim.pde@gmail.com; chanwoo.kim@wisc.edu}
 
\date{\today}
\maketitle

\begin{abstract}
We prove that exponential moments of a fluctuation of the pure transport equation decay \textit{pointwisely} almost as fast as $t^{-3}$ when the domain is any general strictly convex subset of $\R^3$ with the smooth boundary of the diffuse boundary condition. We prove the theorem by establishing a novel $L^1$-$L^\infty$ framework via stochastic cycles.
\end{abstract}

%%%%%%%%%%%%%%%%%%%%%%%%%%%%%%

\section{Introduction and the result of this paper}

An important and active research direction in the mathematical kinetic theory is on the asymptotic behavior of its solutions as $t\rightarrow \infty$ for both the collisional models (e.g. \cite{CKL,DV,Guo10,Mouhot,K_CPDE,SG}) and the collisionless models (e.g. \cite{BMM,Landau, MW, MV}). In this paper, we are interested in a mixing effect of stochastic boundary damping the moments of fluctuation for a simple collisionless model. More precisely, we consider a \textit{free transport} equation in a bounded domain $\O \subset\R^3$, with an initial condition 
$F(t,x, v) |_{t = 0}  = F_0 (x, v)$, 
\begin{equation} \label{equation for F} 
\partial_t F + v \cdot \nabla_{x} F = 0,  \ \  \text{for} \  (t, x, v) \in \mathbb{R}_{+} \times \Omega \times \mathbb{R}^3. 
\end{equation} 
Throughout this paper, we assume the domain is \textit{smooth} and \textit{strictly convex}: there exists a smooth function 
$\xi: \R^3 \rightarrow \R$ such that $\O= \{x \in \R^3: \xi(x) <0\}$ and 
$\sum_{i,j} \p_{i}\p_j \xi(x) \zeta_i \zeta_j \gtrsim |\zeta|^2$ for all $\zeta \in \R^3$ (\cite{GKTT1}). The phase boundary $\gamma:=  \{ (x, v) \in \partial \Omega \times \mathbb{R}^3\}$ is decomposed into the outgoing boundary and incoming boundary $ \gamma_{\pm} := \{ (x, v) \in \partial \Omega \times \mathbb{R}^3, n(x) \cdot v \gtrless 0 \}$ with the outward normal $n(x)$ at $x \in \p\O$.
%\begin{equation}
%\begin{split}
%  \gamma_{\pm} := \{ (x, v) \in \partial \Omega \times \mathbb{R}^3, n(x) \cdot v \gtrless 0 \}.\label{gamma_pm}
%\\&  \gamma_{+} := \{ (x, v) \in \partial \Omega \times \mathbb{R}^3, n(x) \cdot v > 0 \},
%\end{split}
%\end{equation} 
We consider an \textit{isothermal diffusive reflection} boundary condition which is the simplest model among the family of stochastic boundary conditions (see \cite{EGKM,CKL2} for the general boundary conditions)
\begin{equation} \label{diff_F}
F(t, x, v)  %= c_\mu \mu(v) J(t, x):
=c_\mu \mu(v) \int_{n(x)\cdot v_1>0} F(t, x, v_1) \{ n(x) \cdot v_1 \} \dd v_1, \ \ \text{for} \ (t,x,v) \in \R_+ \times \gamma_-.
\end{equation} 
Here, for $c_\mu = \sqrt{2 \pi}$, $c_\mu\mu(v) = c_\mu  \frac{1}{(2 \pi)^{3/2}}\exp\{-   {|v|^2} / {2}\}$ 
stands for the wall Maxwellian distribution of the unit wall temperature. At the molecule level, the boundary condition \eqref{diff_F} corresponds to the Markov process at the boundary (\cite{Spohn}). % and 
%$c_\mu  := \big(\int_{n(x) \cdot v_1>0} \mu(v_1) \{n(x) \cdot v_1\} \dd v_1 \big)^{-1}
%= \sqrt{2\pi}$. 
We set the total mass of the initial datum to be $\mathfrak{M} \times |\O|$, for some $\mathfrak{M} \geq0$:
\begin{equation} \label{cons_m}
 % \iint_{\Omega \times \mathbb{R}^3} F(t, x, v) \dd x \dd v = 
  \iint_{\Omega \times \mathbb{R}^3} F_0 ( x, v) \dd x \dd v= \iint_{\Omega \times \mathbb{R}^3} \mathfrak{M}\mu (v) \dd x \dd v.
\end{equation} 
The choice of $c_\mu= \sqrt{2 \pi}$ formally guarantees a null flux condition at the boundary and the conservation of mass. We are interested in a long time behavior of the fluctuation of $F$ around the equilibrium $\mathfrak{M}\mu (v)$: 
\Be f(t, x, v)=
F(t, x, v) - \mathfrak{M}\mu (v) \ \ \text{where} \ \ 
\mu (v)= 
 \frac{1}{(2 \pi)^{3/2}}e^{-   \frac{|v|^2}{2}}
.\label{pert_f}
\Ee

 Due to its conceptual importance and applications, the mixing effect of the stochastic boundary has been studied in various aspects of the Boltzmann equation. In \cite{Guo10}, Guo establishes a novel $L^2$-$L^\infty$ framework to control an $L^\infty_x$-norm of the Boltzmann equation for all basic boundary conditions (e.g. diffuse reflection, specular reflection, inflow, and bounce-back conditions). In this framework of \cite{Guo10}, an $L^\infty_x$-norm can be controlled directly along the generalized characteristics corresponding to the boundary condition, the bouncing billiard trajectories with stochastic boundary in the case of \eqref{diff_F}, without any differentiability assumption. In \cite{Kim}, Kim constructs initial data of the Boltzmann equation inducing the formation of singularity at the boundary and proves the propagation of such singularity along with the generalized characteristics. In \cite{EGKM}, Esposito-Guo-Kim-Marra construct the stationary solutions of the Boltzmann equation when the boundary temperature can be non-constant. In fact, these solutions are non-equilibrium stationary states since they are not local Maxwellians. They also prove exponentially-fast asymptotical stability of such stationary states under small perturbations in $L^\infty_x$ (\cite{EGKM}). In \cite{CKL,CKL2}, Kim et al. construct strong solutions of Vlasov-Poisson-Boltzmann systems in convex domains with the diffuse reflection boundary and prove exponentially-fast asymptotical stability. One of the keys in their proof is strong control (in some H\"older space) of the electrostatic force field and bouncing characteristics curves with stochastic boundary. We also refer to \cite{Yu} for a probabilistic approach of the boundary problem of a diffuse reflection boundary condition in 1D. 
  
Damping induced solely by the mixing effect of the stochastic boundary is a primary subject of this paper. It is a different mechanism of the phase mixing \textit{without the Boltzmann collision effect}. Perhaps, the most famous result of the phase mixing is the Landau damping, which generally refers to the decay of the moments of the fluctuation or electrostatic force field for the Vlasov-Poisson system without the boundary (\cite{Landau,MW}). Mathematical justification of the nonlinear Landau damping has been a longstanding open question, which is recently settled in the affirmative by Mouhot-Villani in \cite{MV} for the real analytic fluctuation around spatially homogeneous equilibriums (also see \cite{BMM} for the fluctuation in some Gevrey space). On the other hand, the nonlinear Landau damping around spatially inhomogeneous equilibriums is a challenging open problem. We refer to \cite{GZ} for the existence of spatially inhomogeneous steady states which are linearly stable.

Perhaps, the first quantitative study on the asymptotic behavior of the fluctuation can be found in \cite{Yu}, in which Yu proves a decay rate of moments of the fluctuation in $L^\infty$ when the boundary is a 1D slab using a probabilistic approach (of Markov chains of i.i.d. random variables). This approach has been successfully generalized to the multi-D cases of \textit{symmetric} domains (a disk in 2D and a ball in 3D) in \cite{KLT}, in which they obtain an optimal decay rate $t^{-D}$. The \textit{symmetric} assumption of the domains is essential in their proof. Under this condition, the bouncing characteristics can be formed by the independent and identically distributed (i.i.d.) random variables. Moreover, the derivatives of outgoing flux can be bounded with the symmetric condition. In general, such derivatives could blow up in general convex domains (\cite{GKTT1, GKTT2}) and non-convex domains (\cite{Kim}). We also refer to \cite{AG,MS} for the studies on the decay of the fluctuation in $L^1_{x,v}$ when the domains have some symmetry. Recently, there is a very interesting development of the subject toward removing the symmetric assumption (we refer to \cite{Bernou, Lods} for a more complete list of references). In \cite{Bernou}, Bernou develops a method based on Harris' Theorem which is particularly well-suited for problems arising in $L^1$-type of spaces. The work of \cite{Bernou} inspires our work, in particular, at the proof of Proposition \ref{theorem:1}. In \cite{Lods1, Lods}, Lods and Mokhtar-Kharroubi develop a different spectral approach using the Tauberian argument. All works \cite{Bernou, Lods1, Lods} address an asymptotic behavior of the fluctuation itself in some $L^1_{x,v}$-type spaces. 
 
Motivated by the recent progress in the Landau damping (\cite{BMM,GZ,MV}), we are mainly interested in the quantitative asymptotic behavior of the exponential moments of the fluctuation
\Be\notag
  \int_{\R^3} e^{\theta |v|^2} |f(t,x,v)| \dd v  \ \ \ \text{in some } strong \text{ space in $x$ without any differentiability assumption.}  
\Ee 
We emphasize that the strong-in-$x$ control of moments is a key step toward nonlinear problems such as the Vlasov-Poisson systems. The low regularity framework has a significant benefit in the nonlinear boundary problems. We refer to \cite{CKL,CKL2} for the method of control the force field of the Vlasov-Poisson-Boltzmann systems interacting with the diffuse reflection boundary. In this paper, we contribute toward establishing a decay of \textit{exponential moments} of the fluctuation in $L^\infty_x$ with \textit{an almost optimal rate} $\frac{1}{t^{3-}}$ when the domain is a \textit{general strictly convex domain} in 3D.  
\begin{theorem}\label{theorem}
Let $\O$ be smooth and strictly convex. Assume \eqref{cons_m} for any $\mathfrak{M}\geq0$. Assume $\| e^{\theta^\prime |v|^2} f_0\|_{L^\infty_{x,v}}<\infty$ for $0<\theta^\prime<1/2$, and $\| \varphi_4 (\tf) f_0 \|_{L^1_{x,v} }<\infty$, with $\varphi_4 (\tf)$ defined in Definition \ref{def:varphis}. There exists a unique solution $F(t,x,v)= \mathfrak{M} \mu(v) + f(t,x,v)\geq 0$ to \eqref{equation for F} and \eqref{diff_F}, such that $\sup_{t\geq0}\| e^{\theta^\prime |v|^2} f (t)\|_{L^\infty_{x,v}}\leq C \| e^{\theta^\prime |v|^2} f_0\|_{L^\infty_{x,v}}$, and  
\Be
\iint_{\Omega \times \mathbb{R}^3} f (t, x, v) \dd x \dd v 
 = \iint_{\Omega \times \mathbb{R}^3} f_0 ( x, v) \dd x \dd v = 0, \ \ \text{for all } t\geq 0. \label{cons_mass_f} 
\Ee
Moreover, for any $\theta \in [0, \theta^\prime)$, there exists $C_\theta>0$ such that 
\Be
\sup_{x \in \bar{\O}}\int_{\R^3} e^{\theta  |v|^2} |f(t,x,v) |\dd v  \leq C_{\theta}   \langle t\rangle^{ -3}(\ln  \langle t\rangle)^{2},  \ \ \text{for all }  t \geq 0. 
%\ \text{and} \ 0 \leq \theta< \theta^\prime. 
\label{theorem_infty} 
\Ee
Here, we have used a notation $\langle  \cdot \rangle:=  e+ |\cdot|$. 
\end{theorem}
\begin{remark}
In contrast to \cite{KLT}, we do not need any symmetric condition on the domain. 
\end{remark}
\begin{remark}
Without loss of generality, we set $\mathfrak{M}=1$ in the rest of the paper, for the sake of simplicity. Following the same proof of this paper, it is straightforward to prove the result to a $D$-dimension for any $D \in \mathbb{N}$ with different decay rates.  
\end{remark}
%%%%%%%%%%%%%%
\hide
\textcolor{red}{
\begin{remark} 
Although in the proofs we use the strictly convexity of the domain in many places, we believe that this assumption can be removed using more delicate study on the trajectory. In this paper we adopt the assumption to present our idea in a simpler manner. 
\end{remark}
}
\unhide
%%%%%%%%%%%%%%%

We record the equation, initial datum, and the boundary condition for the fluctuation $f$ in \eqref{pert_f}:  
\begin{align} 
\partial_t f + v \cdot \nabla_{x} f = 0,& \ \  \text{for} \  \    (t, x, v) \in \mathbb{R}_{+} \times \Omega \times \mathbb{R}^3,  \label{eqtn_f} \\
 f (t,x,v) |_{t = 0}  = f_0(x,v):= F_0 (x, v) - \mathfrak{M} \mu (v),&  \ \  \text{for} \   \   (x, v) \in \Omega \times \R^3,
\label{init_f}
 \\
 f (t, x, v)  = c_\mu\mu(v)  \int_{n(x) \cdot v_1>0} f(t, x, v_1) \{ n(x) \cdot v_1 \} \dd v_1 ,& \ \  \text{for} \  \   (t,x, v) \in \mathbb{R}_{+} \times \gamma_-. \label{diff_f} 
\end{align}

\noindent \textbf{Notations.} We shall clarify some notations: $A \lesssim_\theta B$ if $A\leq C B$ for a constant $C=C(\theta)>0$ which depends on $\theta$ but is independent on $A,B$; $A \sim B$ if $A \lesssim B$ and $B \lesssim A$;
 $A \leq O(B)$ if $|A| \lesssim B$; 
$\| \cdot \|_{L^1_{x,v}}$  for the norm of $L^1(\O \times \R^3)$; $\| \cdot \|_{L^\infty_{x,v}}$ or $\| \cdot \|_\infty$  for the norm of $L^\infty(\bar{\O} \times \R^3)$; $|g|_{L^1_{\gamma_\pm}}=\int_{\gamma_\pm} |g(x,v) | |n(x) \cdot v| \dd S_x \dd v$; an integration $\int_Y f(y)\dd y$ is often abbreviated to $\int_Y f$, if it is not ambiguous. 

\subsection{Novel $L^1$-$L^\infty$ framework via Stochastic Cycles}
In a broad sense, our argument of the $L^1$-$L^\infty$ framework to prove Theorem \ref{theorem} bears some resemblance to the framework developed in the study of the Boltzmann equation \cite{CKL,Guo10,EGKM}. A foundational idea of our novel $L^1$-$L^\infty$ framework over the whole paper is to \textit{transfer a velocity mixing from the diffusive reflection (\ref{diff_f}) to a spatial mixing through the transport operator}. This idea is realized via \textit{the stochastic cycles}:
\begin{definition}[\cite{CKL,EGKM,Guo10}]\label{def_cycles}Define the backward exit time $t_{\mathbf{b}}$ and the forward exit time $\tf$
%, have appeared: %for $(x, v) \in \Bar{\Omega} \times \mathbb{R}^3$, we define 
%the backward exit time $t_{\mathbf{b}}$ are the forward exit time $\tf$ defined as  
\begin{equation}\label{def_tb} 
    t_{\mathbf{b}}(x, v) := \sup \{s \geq 0: x - \tau v \in \Omega, \  \forall\tau \in [0, s) \},   \ x_{\mathbf{b}}(x, v) := x - t_{\mathbf{b}}(x, v) v, \ (\tf,\xf)(x,v) := (\tb,\xb)(x,-v). 
\end{equation}
We define the stochastic cycles: $t_1 (t, x, v) = t - t_{\mathbf{b}}(x, v)$, $x_1 (x, v) = x_{\mathbf{b}}(x, v) := x - t_{\mathbf{b}}(x, v) v,$
\begin{equation} \label{def:t_k}
    t_k (t, x, v, v_1,..., v_{k-1}) = t_{k-1}% (t, x, v, v_1,..., v_{k-2}) 
    - t_{\mathbf{b}}(x_{k-1}, v_{k-1}),
 \ \  x_k (t, x, v, v_1,..., v_{k-1}) = x_{k-1} %(t, x, v, v_1,..., v_{k-2}) 
    - t_{\mathbf{b}}(x_{k-1}, v_{k-1}) v_{k-1}, 
\end{equation}
where a free variable $v_j \in  \mathcal{V}_j  := \{v_j \in \mathbb{R}^3: n(x_j) \cdot v_j > 0 \}$.
\end{definition}

\begin{lemma}[\cite{CKL,EGKM,Guo10}]
\label{sto_cycle}
 Suppose $f$ solve (\ref{eqtn_f}), (\ref{diff_f}) and $t_* \leq t$. For $g(t, x, v) := \varrho(t)w (v) f (t, x, v)$ with given $\varrho(t)$, $w(v)$,  
 % Let $\varrho(t)$ an increasing smooth function of $t\geq0$. Set, for $f$ solving (\ref{eqtn_f})-(\ref{diff_f}),
%\Be\label{def:h} 
%\Ee 
%Then it solves $\p_t g + v\cdot \nabla_x g = \varrho^\prime w f$ and $g(t, x, v)|_{\gamma_-}  = c_\mu\mu(v) w(v)  \int_{n(x) \cdot v_1>0} g(t, x, v_1)\frac{1}{w(v_1)} \{ n(x) \cdot v_1 \} \dd v_1$.
%\Be
%\begin{split}
%\p_t g + v\cdot \nabla_x g = \varrho^\prime w f ,&  \ \   (t, x, v) \in \mathbb{R}_{+} \times \Omega \times \mathbb{R}^3,\\
%g(t, x, v)  = c_\mu\mu(v) w(v)  \int_{n(x) \cdot v_1>0} g(t, x, v_1)\frac{1}{w(v_1)} \{ n(x) \cdot v_1 \} \dd v_1 ,& \ \  (t,x, v) \in \mathbb{R}_{+} \times \gamma_-. \label{diff_h} 
%\end{split}
%\Ee
%As \eqref{expand_f}, we have a representation 
\begin{align}
    g (t, x, v) 
   & = \mathbf{1}_{t_1 \leq t_*}
    g (t_*, x -  (t-t_*)   v, v)
     \label{expand_h1}
    \\& + \int^t_{\max(t_*, t_{1})} \varrho^\prime(s) w(v) f(s, x-(t-s)v,v)\dd s
\label{expand_h21}
    \\& +  c_\mu w \mu (v)  \int_{\prod_{j=1}^{k} \mathcal{V}_j}   
     \sum\limits^{k-1}_{i=1} 
     \Big\{   \mathbf{1}_{t_{i+1}<t_* \leq t_{i }}  g (t_*, x_{i } -  (t_{i }-t_*)   v_{i }, v_{i})  \Big\}
      \dd  \Sigma ^{k}_{i}
\label{expand_h22}
    \\& + c_\mu w \mu (v)  \int_{\prod_{j=1}^{k} \mathcal{V}_j}   
     \sum\limits^{k-1}_{i=1} 
     \Big\{ \mathbf{1}_{t_* \leq t_{i }}
     \int^{t_i}_{ \max(t_*, t_{i+1})}w(v_i)\varrho^\prime(s) f(s, x_i-(t_i-s)v_i,v_i)\dd s
         \Big\}
      \dd  \Sigma ^{k}_{i} \label{expand_h2}
    \\& +  c_\mu w \mu(v)  \int_{\prod_{j=1}^{k } \mathcal{V}_j}   
    \mathbf{1}_{t_{k } \geq t_* }
    g (t_{k }, x_{k }, v_{k })
     \dd  \Sigma ^{k }_{k }
, \label{expand_h3}
\end{align} 
where
$ \dd  {\Sigma}^{k}_{i}:=\dd \sigma_{k} \cdots  \dd \sigma_{i+1}  \frac{ \dd \sigma_{i}}{c_\mu \mu(v_i)w(v_i)} \dd \sigma_{i-1} \cdots  \dd \sigma_1$, % for $i \in \{1, \cdots, k\}$, 
 with a probability measure $\dd \sigma_j = c_\mu\mu(v_j) \{ n(x_j) \cdot v_j \} \dd v_j$ on $\mathcal{V}_j$.
\end{lemma}

 \subsection{Weighted $L^1$-estimates}
 As the first part of our $L^1$-$L^\infty$ framework, we prove an $L^1$-decay of the fluctuation $f$ as $t\rightarrow \infty$ in Proposition \ref{theorem:1}, following the idea of aperiodic Ergodic theorem (e.g. \cite{Bernou,MT}).   
We prove a key lower bound with a \textit{unreachable defect}, crucially using the stochastic formulation in Lemma \ref{sto_cycle} (see the precise statement in Lemma \ref{lem:Doeblin}): for $f_0  \geq 0$, $t-t_*\gg1$,
 \Be\label{est:Doeblin_rought}
f(t,x,v)\geq  
\mathfrak{m}(x,v)  \{
\|f(t_* ) \|_{L^1_{x,v}}
-   \| \mathbf{1}_{t_\mathbf{f} \gtrsim |t-t_*| %\frac{3T_0}{4}
}  f(t_* )   \|_{L^1_{x,v}}
 \} \ \ \text{for some non-negative function} \  \mathfrak{m} .
\Ee  
This unreachable defect, which stems from small velocity particles in the outgoing flux of the diffuse reflection (\ref{diff_F}), is intrinsic unless the wall Maxwellian $c_\mu\mu(v)$ vanishes around $|v|=0$. 
%%%%%%%%%%%%%
\hide
 \begin{lemma}%[Doeblin's condition]
%\label{lem:Doeblin}
Suppose $f$ solve (\ref{eqtn_f}) with (\ref{diff_f}). Assume $f_0(x,v) \geq 0$ (no need of (\ref{cons_mass_f})). For any $T_0\gg1$ and $N\in \mathbb{N}$ there exists $\mathfrak{m}(x,v)\geq 0$, which only depends on $\O$ and $T_0$ (see (\ref{def:m}) for the precise form), such that 
\Be\label{est:Doeblin}
f(t,x,v)\geq  
\mathfrak{m}(x,v) \Big\{
\iint_{\O \times \R^3}f(t_*,x,v) \dd v \dd x 
-  \iint_{\O \times \R^3} \mathbf{1}_{t_\mathbf{f}(x,v)\geq \frac{3T_0}{4}} f(t_*,x,v)  \dd v \dd x 
\Big\}.
\Ee
%and 
%\Be\label{def_tf}
%\tf(x,v) : = \tb(x,-v).
%\Ee

\end{lemma} \unhide
%%%%%%%%%%%%%%

Next we control the unreachable defect using the weighted $L^1$-estimates. Due to the invariance of $\xb$ and $\xf$ under $v\cdot \nabla_x$, which has been crucially used in construction of the distance function invariant under Vlasov operator in \cite{CKL}, a weight $\varphi(\tf)$ provide an effective dissipation $v\cdot \nabla_x \varphi(\tf)= -\varphi^\prime(\tf)$ for $\varphi^\prime\geq0$, as long as a byproduct term on $\gamma_-$ can be controlled. Inspired by the proof of an $L^1$-trace theorem of \cite{GKTT1}, we derive that 
\begin{lemma}\label{lemma:energy}
Suppose $\varphi (\tau ) \geq0$, $\varphi^\prime \geq0$, and 
\Be\label{cond:varphi} 
\int_1^\infty \tau^{-5}\varphi(\tau) \dd \tau< \infty. 
\Ee
Suppose $f$ solve (\ref{eqtn_f}) and (\ref{diff_f}). Then there exists $C>0$ such that for all $0 \leq t_* \leq t$, 
\Be
\begin{split}
\label{energy_varphi}
  \| \varphi(\tf) f(t) \|_{L^1_{x,v}}
  + \int^{t}_{t_*}
  \| \varphi^\prime(\tf) f  \|_{L^1_{x,v}} 
  +  \int_{t_*}^{t} | \varphi(\tf) f|_{L^1_{\gamma_+}}
  -  \frac{1}{4}  \int^{t}_{t_*} |f |_{L^1_{\gamma_+}}  
  \leq     \| \varphi(\tf) f(t_*) \|_{L^1_{x,v}}+
 C \| f(t_*) \|_{L^1_{x,v}}.
\end{split}
\Ee 
\end{lemma}
It is worth informing beforehand that the exponent $-5$ in (\ref{cond:varphi}) will basically restrict the decay rate of Theorem \ref{theorem}. Some postulation on the wall Maxwellian such as 
%%%%%%%%%%
\hide

$\mu(v)/|v|^r<\infty$ 

\unhide
%%%%%%%%%%
 $\mu(v)/ \langle v \rangle^r<\infty$ 
for some $r>0$ in (\ref{diff_F}) or a similar assumption on the inflow boundary condition would provide faster decay. 

Employing a function $\varphi$ with $\varphi^\prime \rightarrow  \infty$ as $\tau \rightarrow \infty$ (see $\varphi_1$ in \eqref{varphis}), an $L^1$-term majorizes the unreachable defect of the lower bound \eqref{est:Doeblin_rought} with a large factor $\varphi(\frac{3T_0}{4})$. Adding \eqref{est:Doeblin_rought} and \eqref{energy_varphi} with the proper ratio, suggested by the large factor, we establish the uniform estimates of the following energies (see $\varphi_i$'s in \eqref{varphis}), with $\| \mathfrak{m} \|_{L^1_{x,v}} \sim \delta_{\mathfrak{m},T_0}$ (see \eqref{est:m}), 
 \Be\label{|||i}
\vertiii{f}_i:=   \|f  \|_{L^1_{x,v}}
    +    \frac{ 4\delta_{\mathfrak{m}, T_0} }{ \varphi_{i-1} (\frac{3T_0}{4})}
 \|  \varphi_{i-1}(\tf) f\|_{L^1_{x,v}}   +    \frac{ 4\delta_{\mathfrak{m}, T_0} }{T_0  \varphi_{i-1}(\frac{3T_0}{4})}\| 
 \varphi_{i}(\tf)
f\|_{L^1_{x,v}}, \ \ \text{for } \   i=1,4.
   \Ee

Finally we interpolate
$\| 
 \varphi_{1}(\tf)
f\|_{L^1_{x,v}}$ by $\| 
 \varphi_{0}(\tf)
f\|_{L^1_{x,v}}$ and $\| 
 \varphi_{4}(\tf)
f\|_{L^1_{x,v}}$, and using the boundedness of $\vertiii{f}_4$, we prove the $L^1$-decay result (see also the similar result in \cite{Bernou}):
\begin{proposition}\label{theorem:1}
Given the same assumptions of Theorem \ref{theorem}, %for any $0<\e\ll1$,
\Be\label{theorem_1}
\| f(t) \|_{L^1_{x,v}} \lesssim  (\ln\langle t\rangle)^{2}%^{2-\e} 
\langle t\rangle^{ -4}  \{ \| e^{\theta^\prime |v|^2} f_0\|_{L^\infty_{x,v}}+ \| \varphi_4(\tf) f  _0 \|_{ L^1_{x,v}} 
 \} .
\Ee
\end{proposition}
%We note again that the same result can be found in \cite{Bernou}, while we believe that our proof is simpler and more direct than one of \cite{Bernou} regardless the convexity assumption of ours. 

 \subsection{An $L^\infty$-estimate of Moments} 
   We bootstrap the $L^1$-decay secured in Proposition \ref{theorem:1} to the pointwise bound of the moments. Again, the crucial tool is the stochastic cycle representation in Lemma \ref{sto_cycle} for $t_*=0$. In light of \eqref{theorem_1}, we have a natural choice of $\varrho$ so that $\varrho^\prime(t)\lesssim  (\ln\langle t\rangle)^{-2}%^{-(2-\e)}
 \langle t\rangle^{4}$ (see \eqref{varrho}). 
   We first establish the control of the time integration terms of \eqref{expand_h2} (we control \eqref{expand_h21} similarly, after applying the stochastic cycles twice):
\begin{lemma}\label{lem:bound1}
For $i = 2, \cdots , k-1$, $w(v)=e^{\theta |v|^2}$ for $\theta>0$, and a differentiable $\varrho(t)$, we have
\Be\label{bound1:expand_h}
  \Big|\int_{\prod_{j=1}^{k} \mathcal{V}_j}   
  \mathbf{1}_{t_{i+1}<0 \leq t_{i }} 
 \int^{t_i}_{0}w(v_i)\varrho^\prime(s) f(s, x_i-(t_i-s)v_i,v_i)\dd s
      \dd {\Sigma}^{k}_{i}\Big|\lesssim
 \int^t_{0}
       \| \varrho^\prime(s) f(s) \|_{L^1_{x,v}} \dd s.
\Ee
%\begin{equation}
%\int_{\mathcal{V}_1} \cdots \int_{\mathcal{V}_{i-1}} \frac{h (t_{*}, x_{i-1} - (t_{i-1} - t_{*}) v_{i-1}, v_{i-1})}{W %(v_{i-1})} d \sigma_1 \cdots \frac{d \sigma_{i-1}}{\mu(v_{i-1})} \lesssim \|f (t_{*})\|_{L^1},
%\end{equation}
%where $t_{i-1} \geq t_{*} > t_i$.
\end{lemma}
The key idea of the proof is using the change of variables $v_{i-1} \mapsto (\xb(x_{i-1}, v_{i-1}), \tb(x_{i-1},v_{i-1}))$, which has been crucially used in evaluating the boundary singularity in \cite{CKL}. By this change of variables we are able to convert the velocity integral of $\dd \sigma_{j-1}$ into an integration of the spatial variable $x_i- (t_i-s)v_i= \xb(x_{i-1}, v_{i-1}) - (t_{i-1} -\tb(x_{i-1},v_{i-1}) -s) v_i$, while the singularity occurs from its Jacobian when $\tb(x_{i-1},v_{i-1})=0$ (see Lemma \ref{lem:mapV}). We remedy such singularity by applying the change of variables twice for $j=i-1$ and $j=i-2$: among the free variables $\{  \xb(x_{i-1}, v_{i-1}), \tb(x_{i-1},v_{i-1}),  \xb(x_{i-2}, v_{i-2}), \tb(x_{i-2},v_{i-2}) \}$ we utilize $\xb(x_{i-1}, v_{i-1})$ and $\tb(x_{i-2},v_{i-2})$ for the spatial variables $x_i- (t_i-s)v_i= \xb(x_{i-1}, v_{i-1}) - (t_{i-2}  -\tb(x_{i-2},v_{i-2})-\tb(x_{i-1},v_{i-1})-s) v_i$, while we are able to appease singularity from the two change of variables using the integration of $\xb(x_{i-2}, v_{i-2})$ and $\tb(x_{i-1},v_{i-1})$.

 \smallskip
 
Next we control \eqref{expand_h3} by establishing the following estimate:
\begin{lemma}\label{lem:small_largek}There exists $\mathfrak{C}= \mathfrak{C}(\O)>0$ (see \eqref{choice:k} for the precise choice) such that  
\begin{equation} \label{small_largek}
\text{if }  \  k \geq \mathfrak{C}t  \text{ then } 
\sup_{(x,v) \in \bar{\O} \times \R^3}  \Big(\int_{\prod_{j=1}^{k -1} \mathcal{V}_j}   
    \mathbf{1}_{t_{k }(t,x,v,v_1,\cdots, v_{k-1}) \geq 0 }
\dd \sigma_1 \cdots \dd \sigma_{k-1}\Big) \lesssim e^{-t}.
\end{equation}  
\end{lemma}
Similar results have been used in \cite{CKL,Guo10,EGKM} but in this paper, we improve the result (the choice of $k$, in particular) using a sharper bound for the summation of combination from Stirling's formula.

 \smallskip

In the rest of the paper, we collect some basic preliminaries in Section 2; then study the weighted $L^1$-estimates and prove Proposition \ref{theorem:1} in Section 3; and finally prove key results in the $L^\infty$-estimate of moments, and then Theorem \ref{theorem} in Section 4. 
%For the sake of readers' convenience we present the proof of Theorem \ref{theorem} only using Proposition \ref{theorem:1}, Lemma \ref{lem:bound1}, and Lemma \ref{lem:small_largek}. 

\section{Preliminaries} In this section we state basic preliminaries mainly collected from \cite{CKL,EGKM, GKTT1, GKTT2,Guo10}. 
\begin{lemma}[Lemma 9 in \cite{CKL}]\label{lem:mapV}
Suppose $\O$ is an open bounded subset of $\R^3$ and $\p\O$ is smooth. 
\begin{itemize}
\item For $x \in \p\O$, consider a map
\Be\label{mapV}
v \in \{v \in \R^3
: n(x) \cdot v  >  0\}
\mapsto (\xb, \tb):=
(\xb(x,v), \tb(x,v)) \in \partial \Omega \times \mathbb{R}_{+}.
\Ee Then the map (\ref{mapV}) is bijective and has the change of variable formula as   
\begin{equation}\label{jacob:mapV}
    \dd v  =|\tb|  ^{-4}|n (\xb ) \cdot (x - \xb) | \dd \tb \dd S_{\xb }.
\end{equation}
\item  Similarly we have a bijective map 
\Be\label{mapV_f}
v \in \{v \in \R^3
: n(x) \cdot v<0\}
\mapsto (\xf , \tf ):=
(\xf(x,v), \tf(x,v)) \in \partial \Omega \times \mathbb{R}_{+},
  \ \ 
  \text{with} \   \dd v  = |\tf|  ^{-4}|n (\xf) \cdot (x - \xf) | \dd \tf \dd S_{\xf}.
\Ee 
\end{itemize}

\end{lemma}

\begin{lemma}[Lemma 3, Lemma 4 in \cite{CKL}]For any $g$, 
\begin{align}
 \int_{\gamma_{\pm}} \int_0^{t_{\mp}(x,v)} g(x\mp sv,v) |n(x) \cdot v|\dd s \dd v \dd S_x 
 = \iint_{\O \times \R^3} g(y,v) \dd y \dd v,\label{COV}\\
\int_{\gamma_{\pm}}
g(x_{\mp}(x,v),v)
|n(x) \cdot v| \dd v \dd S_x
  = %\int_{\p\O}\int_{n(y) \cdot v<0}
\int_{\gamma_{\mp}}
g(y,v)
|n(y) \cdot v|\dd v \dd S_y.\label{COV_bdry}
\end{align}\hide
\Be\label{COV}
 %\int_{\p\O } \int_{n(x) \cdot v\gtrless 0}
 \int_{\gamma_{\pm}} \int_0^{t_{\mp}(x,v)} g(x\mp sv,v) |n(x) \cdot v|\dd s \dd v \dd S_x 
= \iint_{\O \times \R^3} g(y,v) \dd y \dd v,
\Ee
\Be\label{COV_bdry}
%\int_{\p\O}\int_{n(x) \cdot v>0}
\int_{\gamma_{\pm}}
g(x_{\mp}(x,v),v)
|n(x) \cdot v| \dd v \dd S_x
 = %\int_{\p\O}\int_{n(y) \cdot v<0}
\int_{\gamma_{\mp}}
g(y,v)
|n(y) \cdot v|\dd v \dd S_y.
\Ee\unhide
Here, for the sake of simplicity, we have abused the notations temporarily: $t_-=\tb, x_-= \xb$ and $t_+= \tf, x_+= \xf$.  
\end{lemma} 

\begin{lemma} 
Suppose $f$ solve (\ref{eqtn_f}) and (\ref{diff_f}). For $0 \leq t_* \leq t$,
\begin{align}
\| f(t) \|_{L^1_{x,v}}  \leq \| f(t_*) \|_{L^1_{x,v}} , \label{maximum}\\
\int^{t}_{t_*}  |f(s )|_{L^1_{\gamma_+}}  
 \leq \| f(t_*) \|_{L^1_{x,v}}
+ O( \delta^2 ) \int^{t}_{t_*}  |f(s )|_{L^1_{\gamma_+}}  .\label{trace}
\end{align}
\end{lemma}
\begin{proof}
The bound (\ref{maximum}) is from $ 
\| f(t) \|_{L^1_{x,v}} + \int^t_{t*} \iint_{\gamma_+}|f|-  \int^t_{t*} \iint_{\gamma_-}|f|  = \| f(t_*) \|_{L^1_{x,v}}$, and, due to the choice of $c_\mu$ in (\ref{diff_F}), we have
\Be\notag
\int^t_{t*} \iint_{\gamma_+}|f|-  \int^t_{t*} \iint_{\gamma_-}|f| 
= \int^t_{t*} \iint_{\gamma_+}|f|-\int^t_{t*} \Big|\iint_{\gamma_+}f\Big|
\geq  \int^t_{t*} \iint_{\gamma_+}|f|-  \int^t_{t*} \iint_{\gamma_+}|f|=0.
\Ee 

Next we work on (\ref{trace}) inspired by the proof of the $L^1$-trace theorem in \cite{GKTT2}. Choose $\delta \in (0, t-t_*)$. For $(x,v) \in \gamma_+$,  
\Be \label{|f|1}
|f(s,x,v)|\leq   \underbrace{ \mathbf{1}_{0 \leq s- (t-\delta)< \tb(x,v)}
|f(t-\delta, x- (s- (t-\delta)) v,v)| }_{\eqref{|f|1}_1} +  \underbrace{ \mathbf{1}_{s- (t-\delta) \geq \tb(x,v)}
|f(s- \tb(x,v), \xb(x,v), v)|}_{\eqref{|f|1}_2} . 
\Ee
From (\ref{COV}) and (\ref{maximum}), we have $\int^{t}_{t_*} \int_{\gamma_+} (\ref{|f|1})_1   \leq \| f(t-\delta) \|_{L^1_{x,v}}
\leq \| f(t_*) \|_{L^1_{x,v}}$. Now we consider $\eqref{|f|1}_2$. For $y= \xb(x,v)$, we have
%\Be
$\mathbf{1}_{s- (t-\delta) \geq \tb(x,v)}
= \mathbf{1}_{s- (t-\delta) \geq \tf(y,v)}\leq 
\mathbf{1}_{\delta \geq \tf(y,v)}$ for $s \in [t_*,t]$. %\notag
%\Ee
From the above inequality, further using the Fubini's theorem, (\ref{COV_bdry}), and (\ref{diff_f}) successively, we derive that 
\begin{align} 
&\int^{t}_{t_*} \int_{\gamma_+} (\ref{|f|1})_2   = \int_{\gamma_+}
\int^t_{ t-\delta + \tb(x,v)}|f(s- \tb(x,v), \xb(x,v), v)| \dd s 
\{n(x) \cdot v\} \dd S_x \dd v\notag \\
%&\leq \int_{\gamma_-}
%\mathbf{1} _{\delta> \tf(y,v)}
%\int^t_{t-\delta}|f(s, y, v)| \dd s 
%|n(y) \cdot v| \dd S_y \dd v \\
& \leq \int_{\p\O} \int_{n(y) \cdot v<0}
\mathbf{1} _{\delta> \tf(y,v)}
\int^t_{t-\delta}|f(s, y, v)| \dd s 
|n(y) \cdot v| \dd S_{y} \dd v\notag\\
&\leq \int_{\p\O}
\underbrace{\Big( \int_{n(y) \cdot v<0}
\mathbf{1} _{\delta> \tf(y,v)}c_\mu \mu(v) |n(y) \cdot v|\dd v\Big)}_{(\ref{est:|f|2})_*}
\int^t_{t-\delta}
\int_{n(y) \cdot v_1>0} 
|f(s,y, v_1)| \{n(y) \cdot v_1\} \dd v_1
 \dd s 
 \dd S_{y} .
 \label{est:|f|2}
\end{align}
From $
 |n(y) \cdot v|/|v|^2 \lesssim \tf(y,v)$, we note that $
 \mathbf{1}_{ |n(y) \cdot v| \lesssim \delta |v|^2} \geq \mathbf{1} _{\delta> \tf(y,v)}$. For $\vartheta$ being the angle between $v $ and $n(y)$, 
\begin{equation}
\begin{split} \label{estimate on delta}
    \int_{ |n(y) \cdot v| \lesssim \delta |v|^2} \mu(v ) \{ n(y) \cdot v  \} \dd v  & \leq \int_{ |n(y) \cdot v| \lesssim \delta |v|^2} \mu(v ) \delta |v |^2 
    \dd v , \ \ \ \ \text{by setting $r = |v |$},\\
    & \leq C \int^{\infty}_{0} \delta r^2 e^{- \frac{r^2}{2}} r^2 \dd r \int_{\cos \vartheta < \delta r} \sin \vartheta \dd \vartheta
  \leq C \int^{\infty}_{0} \delta r^2 e^{- \frac{r^2}{2}} r^2 \delta r \dd r \leq C \delta^2.
\end{split}
\end{equation} 
 Then, from \eqref{est:|f|2} and \eqref{estimate on delta}, we conclude $
\int^{t}_{t_*} \int_{\gamma_+} \eqref{|f|1}_2  \leq 
 \eqref{est:|f|2} \lesssim \delta^2  \int^t_{t_*} \int_{\gamma_+} |f|$.\end{proof}
 \begin{lemma}[Lemma 6 in \cite{Guo10}] For a strictly convex domain with a smooth boundary, 
  \Be\label{convex:1}
\max\{ |n(y) \cdot (y-z)|, |n(z) \cdot (y-z)|\}\lesssim |y-z|^2
 \ \text{for all} \  y , z \in \p\O. 
 \Ee
If we further assume that the domain is
strictly convex then there exists $C_\O>0$ such that 
 \Be\label{convex:2}
\min\big( |n(y) \cdot (y-z)|, |n(z) \cdot (y-z)|\big)\geq C_\O |y-z|^2
 \ \text{for all} \  y , z \in \p\O. 
 \Ee
\end{lemma}
 
\section{Weighted $L^1$-Estimates}
The main purpose of this section to prove Proposition \ref{theorem:1}, which happens at the end of this section. We shall start it by settling one of the key cornerstones, Lemma \ref{lem:Doeblin}, the lower bound with the unreachable defect.
\begin{lemma}%[Doeblin's condition]
\label{lem:Doeblin}
Suppose $f$ solve (\ref{eqtn_f}) with (\ref{diff_f}). Assume $f_0(x,v) \geq 0$ (no need of (\ref{cons_mass_f})). For any $T_0\gg1$ and $N\in \mathbb{N}$ there exists $\mathfrak{m}(x,v)\geq 0$, which only depends on $\O$ and $T_0$ (see (\ref{def:m}) for the precise form), such that 
\Be\label{est:Doeblin}
f(NT_0,x,v)\geq  
\mathfrak{m}(x,v) \Big\{
\iint_{\O \times \R^3}f((N-1)T_0,x,v) \dd v \dd x 
-  \iint_{\O \times \R^3} \mathbf{1}_{t_\mathbf{f}(x,v)\geq \frac{3T_0}{4}} f((N-1)T_0,x,v)  \dd v \dd x 
\Big\}.
\Ee
%and 
%\Be\label{def_tf}
%\tf(x,v) : = \tb(x,-v).
%\Ee
 %The unreachable defect, stems from small velocity particles in the outgoing flux of the diffuse reflection (\ref{diff_F}), is intrinsic unless the wall Maxwellian $c_\mu\mu(v)$ vanishes around $|v|=0$. 
\end{lemma} 
\begin{proof}%[\textbf{Proof of Lemma \ref{lem:Doeblin}}]
\textbf{Step 1.} It is standard to derive $f(t,x,v)\geq 0$ from the assumption $f_0(x,v) \geq 0$. For the proof we refer to the standard sequence argument in the proof of Theorem 1 in \cite{Guo10}. Together with (\ref{expand_h1})-(\ref{expand_h3}) for $t= NT_0$, $t_*= (N-1)T_0$, $k=3$, we can derive that 
\begin{align}
f(NT_0, x,v) %&\geq c_\mu \mu(v) \int_{\mathcal{V}_1} 
 %\int_{\mathcal{V}_2}
% \int_{\mathcal{V}_3}  
 %\mathbf{1}_{t_3 \geq (N-1)T_0}
% f(t_3, x_3,v_3) \frac{\dd \sigma_3}{c_\mu \mu(v_3)}
%  \dd \sigma_2\dd \sigma_1
%\notag  \\
&\geq \mathbf{1}_{ \tb(x,v) \leq \frac{T_0}{4}} c_\mu \mu(v) \int_{\mathcal{V}_1} 
 \int_{\mathcal{V}_2}
 \int_{\mathcal{V}_3}  
 \mathbf{1}_{t_3 \geq (N-1)T_0}
 f(t_3, x_3,v_3)  \{n(x_3) \cdot v_3\} \dd v_3
  \dd \sigma_2\dd \sigma_1.\label{Doeblin_1}
\end{align}
%Clearly $\mathbf{1}_{ \tb(x,v) \leq \frac{T_0}{4}} \neq 1$ for all $(x,v) \in \O \times \R^3$. 
 
 Now applying Lemma \ref{lem:mapV} for $v_1 \in \mathcal{V}_1$ and $v_2 \in \mathcal{V}_2$ with (\ref{mapV}) and (\ref{jacob:mapV}), we derive that 
 \Be\label{Doeblin_2}
 \begin{split}
 (\ref{Doeblin_1}) & \geq
   \mathbf{1}_{ \tb(x,v) \leq \frac{T_0}{4}}  c_\mu \mu(v) 
 \int_0^{t-\tb(x,v)}
  \int_{\p\O} 
  \underbrace{ \frac{|n(x_2) \cdot (x_1-x_2)|}{|t_{\mathbf{b}, 1}|^4}  
   \frac{|n(x_1) \cdot (x_1-x_2)|}{ t_{\mathbf{b}, 1}  }  
   c_\mu \mu\Big( \frac{|x_1-x_2|}{t_{\mathbf{b}, 1} }\Big)}_{(\ref{Doeblin_2})_1}\\
  & \times 
     \int_0^{t-\tb(x,v)-t_{\mathbf{b},1}}
  \int_{\p\O}
 \underbrace{    \frac{|n(x_3) \cdot (x_2-x_3)|}{|t_{\mathbf{b}, 2}|^4}  
   \frac{|n(x_2) \cdot (x_2-x_3)|}{ t_{\mathbf{b}, 2}  }  
   c_\mu \mu\Big( \frac{|x_2-x_3|}{t_{\mathbf{b}, 2} }\Big)
   \mathbf{1}_{t_3\geq (N-1)T_0}}_{(\ref{Doeblin_2})_2}
   \\
   & \times 
   \int_{n(x_3) \cdot v_3>0}
   f(t_3,x_3,v_3) \{n(x_3) \cdot v_3\} \dd v_3
    \dd S_{x_3}\dd t_{\mathbf{b},2}
    \dd S_{x_2}\dd t_{\mathbf{b},1}, \ \text{where} \ t_3 = NT_0 - \tb(x,v) - t_{\mathbf{b},1}- t_{\mathbf{b},2}.
 \end{split}\Ee
% where $t_3 = NT_0 - \tb(x,v) - t_{\mathbf{b},1}- t_{\mathbf{b},2}$.
 
 %\smallskip
 
 \textbf{Step 2.} 
To have a positive pointwise lower bound of the integrands of the first two lines of (\ref{Doeblin_2}) we will further restrict integration regimes. Note that $x_1= \xb(x,v)$ is given, and $x_2, x_3$ are free variables. Now we restrict the range of $x_2$ as, for $\delta>0$,  
 \Be\label{def:X_2}
\mathcal{X}_2^\delta := \{ x_2 \in \p\O: |x_1- x_2|> \delta \ \text{and} \  |x_2-x_3|> \delta \},
 \Ee where we pick $\delta$ such that $0 < \delta \ll |\p \O| < \infty$, we can  derive that $|\p \O|/2 \leq |\mathcal{X}_2^\delta| \leq |\p \O|$. 

 For two free variables $t_{\mathbf{b},1}$ and $t_{\mathbf{b},2}$ we use, only inside the proof of Lemma \ref{lem:Doeblin}, two free variables
 \Be\label{def:t+}
 t_+ = t_{\mathbf{b},1} + t_{\mathbf{b},2}\in [0, 
T_0 - \tb(x,v)
 ] \ \  \text{and} \  \ t_- = t_{\mathbf{b},1} - t_{\mathbf{b},2} \in [-(T_0 - \tb(x,v)), T_0 - \tb(x,v)]. 
 \Ee
 Note that the ranges come from $t_3 \geq (N-1) T_0$. Now we restrict the integral regimes of the new variables as 
 \Be\begin{split}\label{def:T+}
 \mathfrak{T}_+^{T_0}:=&
 \Big\{
 t_+ \in  [0, \infty):
 T_0- \tb(x,v) - \min\Big(\tb(x_3,v_3) ,
 \frac{T_0}{4}
\Big)
  \leq  t_+\leq 
 T_0- \tb(x,v)
 \Big\},\\
  \mathfrak{T}_-^{T_0}:= &
  \Big\{
  t_-\in \R:  | t_-| \leq T_0- \tb(x,v) - \min\Big(\tb(x_3,v_3) ,
 \frac{T_0}{4}
\Big)
  \Big\}.
 \end{split}\Ee
As a consequence of (\ref{def:T+}) we will derive  (\ref{min:tb1}) and (\ref{cond:tf}).
 %%%%%%%%%%%%%%
 %%%%%%%%%%%%%%
Firstly, from $\tb(x,v) \leq \frac{T_0}{4}$ in (\ref{Doeblin_2}) and (\ref{def:t+}) 
 \Be\label{min:tb1}
 \begin{split}
\min \big(t_{\mathbf{b},1}, t_{\mathbf{b},2}\big) &=\min \Big( \frac{ t_++ t_-  }{2}
, \frac{ t_+- t_-  }{2}
\Big)
\geq \frac{1}{2}  \{T_0 - \tb(x,v) - \frac{T_0}{4} - \frac{T_0}{4}\}
\geq \frac{T_0}{8},\\
\max \big(t_{\mathbf{b},1}, t_{\mathbf{b},2}\big) & =\max \Big( \frac{ t_++ t_-  }{2}
, \frac{ t_+- t_-  }{2}
\Big)\leq T_0.
\end{split}
 \Ee
Therefore, from \eqref{def:X_2} we exclude the case when $x_1, x_2, x_3$ are too close and from \eqref{def:T+} we exclude the case when either $t_{\mathbf{b},1}$ or $t_{\mathbf{b},2}$ is too small or too large.
\\
  %%%%%%%%%%%%%%
 %%%%%%%%%%%%%%
Secondly, we prove (\ref{cond:tf}). Note that if $t_+ \in \mathfrak{T}_+^{T_0}$ then 
 $(N-1)T_0
  \leq t_3=
 NT_0 - \tb(x,v) - t_+\leq 
 (N-1)T_0 
 +\min \{\tb(x_3,v_3) ,
 \frac{3T_0}{4}
\}$.
This implies that, 
\Be \begin{split}\label{cond:tf}
  \text{if} \  \ \tf(y,v_3)= t_3- (N-1)T_0= T_0 - \tb(x,v) -t_+ \in %\Big[0, \min\Big(\tb(x_3,v_3) , \frac{3T_0}{4}\Big)\Big]
  \Big[0, \frac{3T_0}{4}\Big] \ 
  \ \text{then} \ \  y=X((N-1)T_0;t_3,x_3,v_3),
 \end{split}
\Ee
where we have use an observation $\tf(y,v_3) \leq \tb (x_3,v_3)$ since $x_3= \xf(y,v_3)$.

  \smallskip
  
 \textbf{Step 3.} For (\ref{Doeblin_2}), we adopt the new variables (\ref{def:t+}), and apply the restriction of integral regimes in (\ref{def:X_2}) and (\ref{def:T+}). %Note that, from (\ref{def:T+}),  
% \Be\label{max:tb1}
%\max \big(t_{\mathbf{b},1}, t_{\mathbf{b},2}\big)   = \max \Big( \frac{ t_++ t_-  }{2}
%, \frac{ t_+ - t_-  }{2}
%\Big) \leq T_0. 
 %\Ee
Recall \eqref{convex:2} from the convexity of the domain. From (\ref{min:tb1}) and (\ref{convex:2}), we derive that 
$$
(\ref{Doeblin_2})_i  \geq \frac{C_\O|x_i-x_{i+1}|^2}{T_0^4}
\frac{C_\O|x_i-x_{i+1}|^2}{T_0 } \frac{1}{2 \pi} e^{- \frac{|x_i-x_{i+1}|^2}{2(T_0/8)^2}}
\geq \frac{C_\O^2\delta^4}{2 \pi T_0^5}e^{- \frac{32 \text{diam}(\O)^2}{T_0^2}}, \ \ \text{for} \ i=1,2.
 % \\
%(\ref{Doeblin_2})_2 & \geq\frac{C_\O|x_2-x_3|^2}{T_0^4}
%\frac{C_\O|x_2-x_3|^2}{T_0 } \frac{1}{2 \pi} e^{- \frac{|x_2-x_3|^2}{2(T_0/8)^2}}
%\geq \frac{C_\O^2\delta^4}{2 \pi T_0^5}e^{- \frac{32 \text{diam}(\O)^2}{T_0^2}}.
 $$
Here $\text{diam}(\O)= \sup_{x,y \in \bar{\O}}|x-y|<\infty$. Finally we get 
 \begin{align}
 (\ref{Doeblin_2}) \geq  \mathbf{1}_{ \tb(x,v) \leq \frac{T_0}{4}} 
\frac{C_\O^4\delta^8}{(2 \pi)^2 T_0^{10}}e^{- \frac{64 \text{diam}(\O)^2}{T_0^2}}  &c_\mu
  \mu(v)  \int_{\p\O} \dd S_{x_3} \int_{n(x_3) \cdot v_3>0}
 \dd v_3\{n(x_3) \cdot v_3\}\int_{\mathfrak{T}_+^{T_0}} \dd t_+\notag \\
& \times
\int_{\mathcal{X}_2^\delta} \dd S_{x_2}
\int_{\mathfrak{T}_-^{T_0}} \dd t_-  
   f(
   NT_0 -\tb(x,v) - t_+
   ,x_3,v_3)  \notag\\
  \geq  \mathbf{1}_{ \tb(x,v) \leq \frac{T_0}{4}} 
\frac{C_\O^4\delta^8}{(2 \pi)^2 T_0^{10}}e^{- \frac{64 \text{diam}(\O)^2}{T_0^2}} &c_\mu
  \mu(v)
  |\mathcal{X}_2^\delta| T_0
   \int_{\p\O} \dd S_{x_3} \int_{n(x_3) \cdot v_3>0}\dd v_3 \{n(x_3) \cdot v_3\} 
  \notag
  \\
& \times
  \int^{ T_0- \tb(x,v)}_{   
 T_0- \tb(x,v) - \min\Big(\tb(x_3,v_3) ,
 \frac{T_0}{4}
\Big)} \dd t_+  f(
   NT_0 -\tb(x,v) - t_+
   ,x_3,v_3).\label{lower1}
 \end{align} 
 
Now we focus on the integrand of (\ref{lower1}). Note that 
%\Be\notag
$(NT_0 - \tb(x,v) - t_+)- (N-1)T_0= T_0 - \tb(x,v) - t_+
\in \Big[0, \min\Big(\tb(x_3,v_3) ,
 \frac{T_0}{4}
\Big)\Big]$. 
%\Ee
Therefore 
\Be\label{lower2}
(\ref{lower1})
   = 
    \int^{ T_0- \tb(x,v)}_{   
 T_0- \tb(x,v) - \min\{\tb(x_3,v_3) ,
 \frac{T_0}{4}
\}} f((N-1)T_0, x_3 - (T_0-\tb(x,v) - t_+)v_3,v_3) \dd t_+.
\Ee
Note that, from (\ref{cond:tf}), $\tf( x_3 - (T_0-\tb(x,v) - t_+)v_3,v_3) \in \big[0, \frac{3T_0}{4}\big]$. Now applying (\ref{COV}), we conclude that 
 \Be
 (\ref{Doeblin_2})\geq 
  \mathbf{1}_{ \tb(x,v) \leq \frac{T_0}{4}} 
\frac{C_\O^4\delta^8}{(2 \pi)^2 T_0^{10}}e^{- \frac{64 \text{diam}(\O)^2}{T_0^2}}  c_\mu
  \mu(v)
  |\mathcal{X}_2^\delta| T_0
  \iint_{\O \times \R^3} \mathbf{1}_{\tf(y,v)  \in [0, \frac{3T_0}{4}]} 
  f((N-1)T_0, y,v) \dd v \dd y.  \notag
 \Ee
We conclude (\ref{est:Doeblin}) by setting 
 \Be\label{def:m}
 \mathfrak{m} (x,v):=  \mathbf{1}_{ \tb(x,v) \leq \frac{T_0}{4}} 
 (2 \pi)^{-2}{C_\O^4\delta^8} T_0^{-9}\exp(-  64 \text{diam}(\O)^2T_0^{-2}) 
% e^{- \frac{64 \text{diam}(\O)^2}{T_0^2}}  
  |\mathcal{X}_2^\delta|  c_\mu
  \mu(v)
 .
  \Ee Recall that $\mathcal{X}_2^\delta$ and $\delta$ is defined in (\ref{def:X_2}).
\end{proof}
 
 An immediate consequence of Lemma \ref{lem:Doeblin}, as in \cite{Bernou}, follows.
  \begin{proposition}\label{prop:Doeblin}
 Suppose $f$ solve \eqref{eqtn_f} and (\ref{diff_f}), and satisfy (\ref{cons_mass_f}). Then for all $T_0\gg1$, $0<\delta\ll 1$, and $N \in \mathbb{N}$ 
  \Be\label{L1_coerc}
   \|f(NT_0)\|_{L^1_{x,v}}  \leq  (1-\| \mathfrak{m}\|_{L^1_{x,v}} )   \|f((N-1)T_0)\|_{L^1_{x,v}} 
   + 2 \| \mathfrak{m}\|_{L^1_{x,v}}  \| \mathbf{1}_{\tf\geq \frac{3T_0}{4}} f((N-1)T_0)\|_{L^1_{x,v}} .
 \Ee Here, with $\mathcal{X}_2^\delta$ in \eqref{def:X_2},
 \Be\label{est:m}
  \| \mathfrak{m} \|_{L^1_{x,v}}= \delta_{\mathfrak{m},T_0}\sim(2 \pi)^{-2} C_\O^4\delta^8T_0^{-8} \exp(-  64 \text{diam}(\O)^2T_0^{-2})   |\mathcal{X}_2^\delta| |\p\O|. 
 \Ee
 \end{proposition}
 \begin{proof}%[\textbf{Proof of Proposition \ref{prop:Doeblin}}]
 Decompose 
 \begin{align*}
 f((N-1)T_0,x,v) &=
 f_{ N-1 ,+}(x,v) - f_{ N-1 ,-}(x,v)\\
& := \mathbf{1}_{f((N-1)T_0,x,v)  \geq 0} |f((N-1)T_0,x,v) |- \mathbf{1}_{f((N-1)T_0,x,v)  < 0} |f((N-1)T_0,x,v) |.\end{align*}
Let $f_{\pm}(s,x,v)$ solve (\ref{eqtn_f}) for $s \in [ (N-1)T_0,NT_0]$ with the initial data $f_{N-1,+}$ and $f_{N-1,-}$ at $s=(N-1)T_0$, respectively. Now we apply Lemma \ref{lem:Doeblin} to each $f_{\pm}(t,x,v)$ and conclude (\ref{est:Doeblin}) for both $f= f_+$ and $f=f_-$ respectively. We also note that $\iint_{\O \times \R^3} f((N-1)T_0,x,v) \dd x \dd v=\iint_{\O \times \R^3} f_{N-1,+}(x,v) \dd x \dd v- \iint_{\O \times \R^3} f_{N-1,-}(x,v) \dd x \dd v =0$ implies $\iint_{\O \times \R^3} f_{N-1, \pm}( x,v)  \dd x \dd v =\frac{1}{2} \iint_{\O \times \R^3} |f((N-1)T_0,x,v)| \dd x \dd v$. Then we derive that 
 \begin{align}
 f_{\pm}(NT_0,x,v) &\geq \mathfrak{m}(x,v) \iint f_{N-1,\pm}(x,v) \dd x \dd v -  
  \mathfrak{m}(x,v) \iint_{\O \times \R^3} \mathbf{1}_{\tf(x,v) \geq \frac{3T_0}{4}} f_{N-1,\pm} (x,v ) \dd x \dd v\notag \\
\geq \mathfrak{l}(x,v) :=& \frac{\mathfrak{m}(x,v)}{2}\iint_{\O \times \R^3} |f((N-1)T_0  )|  
  - \mathfrak{m}(x,v) \iint_{\O \times \R^3} \mathbf{1}_{\tf(x,v) \geq \frac{3T_0}{4}} |f((N-1)T_0 )|.\label{lowerB:f}
 \end{align}  
 Then we deduce that 
\Be \notag
 \begin{split}
 |f(NT_0,x,v)| 
& = |f_+(NT_0,x,v)- \mathfrak{l}(x,v) - f_-(NT_0,x,v) + \mathfrak{l}(x,v)|
\\& \leq  |f_+(NT_0,x,v)- \mathfrak{l}(x,v)| +| f_-(NT_0,x,v)- \mathfrak{l}(x,v)| 
 \leq   f_+(NT_0,x,v)+f_-(NT_0,x,v) -2 \mathfrak{l}(x,v).
 \end{split}
 \Ee
Note that $f_+(NT_0,x,v)+f_-(NT_0,x,v)$ solves (\ref{eqtn_f}) with the initial datum $f_{N-1,+} + f_{N-1,-}=|f((N-1)T_0,x,v)|$ at $(N-1)T_0$. Then using (\ref{cons_mass_f}) and taking an integration to (\ref{lowerB:f}) over $\O \times \R^3$, we derive (\ref{L1_coerc}).

For (\ref{est:m}) it suffices to bound $\| \mathbf{1}_{\tb(x,v) \leq \frac{T_0}{4}}  
c_\mu  \mu(v)\|_{L^1_{x,v}}$. From (\ref{COV}) and $\tb(x-sv,v) = \tb(x,v)-s$, 
\Be\notag
\begin{split}
&\| \mathbf{1}_{\tb(x,v) \leq \frac{T_0}{4}}  c_\mu \mu(v) \|_{L^1_{x,v}}  %= 
 %  \int_{\p\O } \int_{n(x) \cdot v>0} \int_0^{\tb(x,v)} 
 %  \mathbf{1}_{\tb(x-sv,v) \leq \frac{T_0}{4}}
  %c_\mu  \mu(v)\{n(x) \cdot v\}\dd s \dd v \dd S_x \\
  % &
  =  \int_{\p\O } \int_{n(x) \cdot v>0} \int_{ \max\{0,\tb(x,v)- \frac{T_0}{4}\}}  ^{\tb(x,v)} 
 c_\mu \mu(v) \{n(x) \cdot v\}\dd s \dd v \dd S_x \\
  &=  \int_{\p\O } \int_{n(x) \cdot v>0} 
\Big(  \mathbf{1}_{\tb(x,v) \leq \frac{T_0}{4}} \int_0^{\tb(x,v)}    \dd s
+\mathbf{1}_{\tb(x,v) \geq \frac{T_0}{4}}\int^{\tb(x,v)}_{\tb(x,v)- \frac{T_0}{2}}  \dd s \Big) 
 c_\mu \mu(v) \{n(x) \cdot v\}\dd v \dd S_x \sim T_0 |\p\O|.  
\end{split}
\Ee
Combining the above bound with (\ref{def:m}), we conclude (\ref{est:m}).
\end{proof}

Next, we prove an important result, Lemma \ref{lemma:energy}, which will be used frequently in this paper.
\hide
\begin{lemma}\label{lemma:energy}
Suppose $\varphi (\tau ) \geq0$, $\varphi^\prime \geq0$, and 
\Be\label{cond:varphi} 
\int_1^\infty \tau^{-5}\varphi(\tau) \dd \tau< \infty. 
\Ee
Suppose $f$ solve (\ref{eqtn_f}) and (\ref{diff_f}). Then there exists $C>0$ such that for all $0 \leq t_* \leq t$, 
\Be
\begin{split}
\label{energy_varphi}
  \| \varphi(\tf) f(t) \|_{L^1_{x,v}}
  + \int^{t}_{t_*}
  \| \varphi^\prime(\tf) f  \|_{L^1_{x,v}} 
  +  \int_{t_*}^{t} | \varphi(\tf) f|_{L^1_{\gamma_+}}
  -  \frac{1}{4}  \int^{t}_{t_*} |f |_{L^1_{\gamma_+}}  
  \leq     \| \varphi(\tf) f(t_*) \|_{L^1_{x,v}}+
 C \| f(t_*) \|_{L^1_{x,v}}.
\end{split}
\Ee 
\end{lemma}\unhide
 \begin{proof}[\textbf{Proof of Lemma \ref{lemma:energy}}] 
 Note that in the sense of distribution 
%\Be\notag
$[\p_t + v\cdot \nabla_x](\varphi(\tf) |f|) = \varphi^\prime(\tf) v\cdot \nabla_x \tf |f|
= -\varphi^\prime(\tf)| f|$. 
%\Ee
From this equation and (\ref{diff_f}), we derive that 
\begin{align}
&\| \varphi(\tf) f (t) \|_{L^1_{x,v}} + \int^t_{t_*} \| \varphi^\prime(\tf) f (s) \|_{L^1_{x,v}}
+ \int^t_{t_*} \int_{\gamma_+} \varphi(\tf) |f| \dd v \dd S_x
\leq  \ \| \varphi(\tf) f (t_*) \|_{L^1_{x,v}}
\notag
\\&
+ \int^t_{t_*} 
\int_{\p\O}
\int_{n(x) \cdot v<0}
 \varphi(\tf) c_{\mu}\mu(v)|n(x) \cdot v|
\int_{n(x)\cdot v_1>0}| f(s,x,v_1)| 
\{n(x) \cdot v_1\}\dd v_1
\dd v
\dd S_x \dd s
.\label{rho.f}
\end{align}
We only need to consider (\ref{rho.f}) with the corresponding $\varphi(\tf)$. We prove the following claim: If (\ref{cond:varphi}) holds then $\sup_{x \in \p\O} \int_{n(x) \cdot v<0}
 \varphi(\tf)(x,v) c_{\mu}\mu(v)|n(x) \cdot v| \dd v\lesssim 1.$ From the claim (\ref{trace}), we conclude (\ref{energy_varphi}), through, for $C>1$,   
\Be
\begin{split}\notag
(\ref{rho.f})  \leq C \int^t_{t_*} 
\int_{\gamma_+} |f(s,x,v_1)| \{n(x) \cdot v_1\} \dd v_1 \dd S_x
\dd s  \leq C \| f((N-1)T_0) \|_{L^1_{x,v}} + \frac{1}{4}  \int^{NT_0}_{(N-1)T_0} |f(s)|_{L^1(\gamma_+)}.
\end{split}\Ee 

For $0<\delta\ll1$, we split $\int_{n(x) \cdot v<0}
 \varphi(\tf)(x,v) c_{\mu}\mu(v)|n(x) \cdot v| \dd v$ into two parts: integration over the regimes of $\tf\leq\delta$ and $\tf>\delta$ respectively. When $\tf\leq\delta$, from \eqref{convex:1}, we derive that $|n(x) \cdot v|/|v|^2
 \lesssim \tf \leq \delta$. Then we bound 
  \Be
 \begin{split}\label{int:rho_1}
 \int_{n(x) \cdot v<0} \mathbf{1}_{\tf\leq\delta}
  \varphi(\tf)(x,v) c_{\mu}\mu(v)|n(x) \cdot v| \dd v
  \lesssim \delta\varphi(\delta) \int_{\R^3} 
|v|^2  \mu(v) 
   \dd v \lesssim 1. 
\end{split} \Ee
 Now we focus on the integration over the regimes of $\tf>\delta$. From (\ref{mapV_f}) we derive that $\int_{n  \cdot v<0} \varphi(\tf)  c_{\mu}\mu(v)|n  \cdot v| \dd v$ equals 
\Be \label{int:rho_2} 
c_\mu\int_{\p\O} \int_\delta^\infty 
\varphi(\tf) \mu\Big( \frac{|x-\xf|}{\tf}\Big) \frac{|n(x) \cdot (x-\xf)|^2}{|\tf|^5}  
\dd \tf \dd S_{\xf}. 
\Ee
From (\ref{convex:1}) and (\ref{cond:varphi}), we derive that 
 $(\ref{int:rho_2})   \lesssim 
\int_\delta^\infty
\frac{\varphi(\tf)}{|\tf|^5} 
\int_{\p\O} | x-\xf |^4 e^{-\frac{|x-\xf|^2}{2|\tf|^2}}\dd S_{\xf}
\dd \tf
\lesssim \int_\delta^\infty \frac{\varphi(\tf)}{|\tf|^5} \dd \tf\lesssim 1$. Together with above bound and (\ref{int:rho_1}) we prove our claim. \end{proof}
We will use the following $\varphi$'s inspired from \cite{Bernou}. 
 \begin{definition}\label{def:varphis}
 For $\delta>0$,
  \Be\label{varphis}
  \begin{split}
% \varphi_{log}(\tau) :=  
 \varphi_0(\tau)&:=(\ln (e+1))^{-1} \ln(e + \ln (e+ \tau))% \mathfrak{log} (e+ \tau)
 , \  \ 
 \varphi_1(\tau):= (e\ln (e+1))^{-1}(e+ \tau) \ln(e + \ln (e+ \tau)), \\
\varphi_3 (\tau)&:=e^{-3}(\tau+e)^3  \big(\ln (\tau+e)\big)^{-(1+\delta)}, \ \ 
\varphi_4 (\tau):=e^{-4}(\tau+e)^4  \big(\ln (\tau+e)\big)^{-(1+\delta)} .
 \end{split}
 \Ee First, we check $\varphi_i$ satisfies (\ref{cond:varphi}) for $i=0,1,2,3,4$: for example, for $\delta>0$
%\Be
%\begin{split}\notag
$\int_1^\infty  \tau^{-5}e^{-4}(\tau+e)^4  \big(\ln (\tau+e)\big)^{-(1+\delta)}  \dd \tau
\lesssim 1+ \int_{10}^\infty (\tau+e)^{-1}\big(\ln (\tau+e)\big)^{-(1+\delta)}  \dd \tau \lesssim  \int_1^\infty \frac{1}{s^{1+\delta}}\dd s$
%\end{split}\Ee
, with $s= \ln(\tau+e)$.

 %%%%%%%%%%
 %%%%%%%%%%
 %%%%%%%%%%
 \hide
 Clearly from (\ref{def:log'}) all the functions satisfy (\ref{cond:varphi}) and
 \begin{align}\label{varphis'}
  \varphi_1^\prime (\tau)=  \mathfrak{log}(\tau)
+ \mathbf{1}_{\tau \geq e^2} + \mathbf{1}_{\tau < e^2} e^{-2} \tau , \ \ 
  \varphi_4^\prime (\tau)= 4 \tau^3 \mathfrak{log}(\tau)
+ \mathbf{1}_{\tau \geq e^2} 4 \tau^3+ \mathbf{1}_{\tau < e^2} 4e^{-2} \tau^4.
 \end{align}
\unhide
 %%%%%%%%%%
 %%%%%%%%%%
 %%%%%%%%%%
Second, we notice that 
 \Be\label{varphi|0}
 \varphi_i (0)=1 \ \ \text{for} \ i=0,1,3,4.
 \Ee
 
Finally, we check
   \Be \label{phi'}
   \begin{split}
&   \varphi_1^\prime(\tau)   =
   (e \ln (e+1))^{-1} \{\ln ( e + \ln ( e + \tau)) +   ( e + \ln ( e + \tau))^{-1}\}
   \geq    (e \ln (e+1))^{-1} \varphi_0 (\tau),
   \ \
    \varphi_0^\prime(\tau) \geq0,
\\& \varphi_4^\prime(\tau) 
=\big(4- \frac{1+\delta}{\ln (\tau+e)}\big)e^{-4} (\tau+e)^3  \big(\ln (\tau+e)\big)^{-(1+\delta)}  \geq 
 \varphi_3 (\tau),\ \
 \varphi_3^\prime (\tau) 
 \geq 0.
\end{split}   \Ee 

 \end{definition}

 \begin{proposition}\label{prop:energy}Choose $T_0>10$ such that \Be\label{cond:T0}
  4 C(2+ T_0) T_0^{-1}   \Big( \varphi_i(\frac{3T_0}{4})\Big)^{-1}
\leq \frac{1}{2} \ \ \text{for } \  i=0,3.
\Ee
For all $N \in \mathbb{N}$ and $i \in \{1,4\}$, 
   \Be\label{energyi}
   \begin{split}   
    & \| f(NT_0) \|_{L^1_{x,v}}
    +    \frac{ 4\delta_{\mathfrak{m}, T_0} }{ \varphi_{i-1} (\frac{3T_0}{4})}
    \Big\{
 \| \varphi_{i-1} (\tf) f  (NT_0)\|_{L^1_{x,v}}  +    \frac{ 1}{T_0 }\| \varphi_i(\tf) f(NT_0) \|_{L^1_{x,v}}
 +   \frac{1}{2T_0  }\int^{NT_0}_{(N-1)T_0} |f|_{L^1_{\gamma_+}}\Big\}
 \\
  \leq& \  \eqref{energyi}_* \times 
    \|f((N-1)T_0)\|_{L^1_{x,v}}  
    + \frac{4\delta_{\mathfrak{m}, T_0} }{ \varphi_{i-1} (\frac{3T_0}{4})}
    \Big\{ \frac{3}{4} \|   \varphi_{i-1}(\tf)f((N-1)T_0)\|_{L^1_{x,v}} +  \frac{1 }{T_0  }
 \| \varphi_i(\tf) f((N-1)T_0) \|_{L^1_{x,v}}\Big\},   
   \end{split}
   \Ee
 with $\eqref{energyi}_*:= (1-  \delta_{\mathfrak{m},T_0} \{
1
 -  \frac{ 4 C(2+ T_0)  }{T_0 \varphi_{i-1}(\frac{3T_0}{4})} \}
 )$ where $\delta_{m,T_0}$ is defined in \eqref{est:m}. 
\end{proposition}
 \begin{proof} 
 As key steps we will repeatedly apply Lemma \ref{lemma:energy} with $\varphi_i$'s in (\ref{varphis}). Applying Lemma \ref{lemma:energy} to $f(t,x,v)$, solving (\ref{eqtn_f}) and (\ref{diff_f}), with $\varphi_i$ for $i=0,1,4$ in (\ref{varphis}), and using \eqref{varphi|0}, %$\tf\equiv0$ on $\gamma_+$ and $\mathfrak{log}(0)=1$.  
we derive that, for $i=1,4$, 
\begin{align}
  \|  \varphi_{i-1}(\tf) f(NT_0) \|_{L^1_{x,v}} +   \frac{3}{4} \int^{NT_0}_{t_*} |f|_{L^1_{\gamma_+}}  \leq  \|  \varphi_ {i-1}(\tf)f(t_*) \|_{L^1_{x,v}}  + C\| f(t_*) \|_{L^1_{x,v}} , \ \ \text{for} \  (N-1)T_0 \leq t_*\leq NT_0,\label{energy:log} \\
% Applying Lemma \ref{lemma:energy} with $\varphi_i$ for $i=1,4$ in (\ref{varphis}) and using \eqref{varphi|0}, we obtain that  
%and, for $i=1,4$, 
  \| \varphi_i(\tf) f(NT_0) \|_{L^1_{x,v}} + \int^{NT_0}_{(N-1)T_0}\{
 \| \varphi_i^\prime(\tf) f \|_{L^1_{x,v}}
 +  \frac{3}{4}  |f|_{L^1_{\gamma_+}}
\} 
 \leq  \| \varphi_i(\tf) f((N-1)T_0) \|_{L^1_{x,v}}  + C\| f((N-1)T_0) \|_{L^1_{x,v}}.\label{energy:phii}
\end{align}
From \eqref{maximum}, \eqref{phi'} and \eqref{energy:log}, we derive that, for $i=1,4$, 
\Be \notag
\int^{NT_0}_{(N-1)T_0}
 \| \varphi_i^\prime(\tf) f \|_{L^1_{x,v}}  \geq \int^{NT_0}_{(N-1)T_0}
 \| \varphi_{i-1}(\tf) f (t_*)\|_{L^1_{x,v}} \dd t_* 
\geq T_0 \| \varphi_{i-1}(\tf) f(NT_0) \|_{L^1_{x,v}} 
 - C T_0 \| f((N-1)T_0) \|_{L^1_{x,v}}.
\Ee 
From the above bound and (\ref{energy:phii}), we conclude that, for $i=1,4$,  
 \Be\label{energy:phi1}
 \begin{split}
& \| \varphi_i(\tf) f(NT_0) \|_{L^1_{x,v}} + 
T_0
 \| \varphi_{i-1}(\tf) f  (NT_0)\|_{L^1_{x,v}}
 + \frac{3}{4} \int^{NT_0}_{(N-1)T_0} |f|_{L^1_{\gamma_+}}
 \\
 &
 \leq  \| \varphi_i(\tf) f((N-1)T_0) \|_{L^1_{x,v}}  + C(1+ T_0)\| f((N-1)T_0) \|_{L^1_{x,v}} 
 .
\end{split} \Ee
  
  Now we combine (\ref{L1_coerc}) with \eqref{energy:log}-\eqref{energy:phi1}. From (\ref{L1_coerc}), $\mathbf{1}_{\tf \geq \frac{3T_0}{4}}%= \mathbf{1}_{\tf \geq \frac{3T_0}{4}} \frac{\varphi_0(\tf)}{\varphi_0(\tf)}
  \leq \big( \varphi_{i-1}(\frac{3T_0}{4}) \big)^{-1} \varphi_{i-1}(\tf)$, with $\delta_{\mathfrak{m}, T_0}$ in (\ref{est:m}),
    \Be\label{L1_coerc'}
  \|f(NT_0)\|_{L^1_{x,v}}  \leq  (1-\delta_{\mathfrak{m}, T_0} )   \|f((N-1)T_0)\|_{L^1_{x,v}} 
   +  2\delta_{\mathfrak{m}, T_0}  \Big( \varphi_{i-1}(\frac{3T_0}{4}) \Big)^{-1} \| \varphi_{i-1}(\tf)f((N-1)T_0)\|_{L^1_{x,v}} .
 \Ee   
% \Be\label{est:m}
  %\| \mathfrak{m} \|_{L^1_{x,v}} \leq \delta_{\mathfrak{m},T_0}:= \frac{C_\O^4\delta^8}{(2 \pi)^2 T_0^{8}}e^{- \frac{64 \text{diam}(\O)^2}{T_0^2}}   |\mathcal{X}_2^\delta| |\p\O|. 
 %\Ee
 
For $i=1,4$, from $(\ref{L1_coerc'})+
  \frac{4 \delta_{\mathfrak{m}, T_0} }{T_0\varphi_{i-1} (\frac{3T_0}{4})}
\big\{
(\ref{energy:log})|_{t=NT_0}+(\ref{energy:phi1})
\big\}
,$ and $T_0>0$ in (\ref{cond:T0}), we deduce (\ref{energyi}).\hide

      \smallskip

      \smallskip
   
   \textbf{Step B.}  As deriving (\ref{energy:log}), we apply Lemma \ref{lemma:energy} and obtain that, for $(N-1)T_0 \leq t_*\leq NT_0$,
\Be\label{energy:phi3}
 \| \varphi_3(\tf) f(NT_0) \|_{L^1_{x,v}} +   \frac{3}{4} \int^{NT_0}_{t_*} |f|_{L^1_{\gamma_+}}  \leq  \| \varphi_3(\tf) f(t_*) \|_{L^1_{x,v}}  + C\| f(t_*) \|_{L^1_{x,v}} .
  \Ee
   
 Recall the estimate (\ref{energy:phii}) with $i=4$. From (\ref{phi'}) and (\ref{energy:phi3}), 
 \Be
 \begin{split}\notag
 &\int^{NT_0}_{(N-1)T_0}
 \| \varphi_4^\prime(\tf) f \|_{L^1_{x,v}}\geq \int^{NT_0}_{(N-1)T_0}
 \| \varphi_3(\tf) f (t_*)\|_{L^1_{x,v}} \dd t_*\\
 &\geq \int^{NT_0}_{(N-1)T_0}
 \Big(\|  \varphi_3 (\tf) f(NT_0) \|_{L^1_{x,v}} 
 %+ \int^{NT_0}_{t_*} |f|_{L^1_{\gamma_+}}
 - C \| f((N-1)T_0) \|_{L^1_{x,v}}
 \Big) \dd t_*\\
 &=T_0 \| \varphi_3(\tf) f(NT_0) \|_{L^1_{x,v}} 
 - C T_0 \| f((N-1)T_0) \|_{L^1_{x,v}}.
 \end{split}
 \Ee 
 From the above bound and (\ref{energy:phii}) with $i=4$, we conclude that 
    \Be\label{energy:phi4}
 \begin{split}
& \| \varphi_4(\tf) f(NT_0) \|_{L^1_{x,v}} + 
T_0
 \| \varphi_3(\tf) f  (NT_0)\|_{L^1_{x,v}}
 + \frac{3}{4} \int^{NT_0}_{(N-1)T_0} |f|_{L^1_{\gamma_+}}
 \\
 &
 \leq  \| \varphi_4(\tf) f((N-1)T_0) \|_{L^1_{x,v}}  + C(1+ T_0)\| f((N-1)T_0) \|_{L^1_{x,v}} 
 .
\end{split} \Ee

From (\ref{L1_coerc}) and $\mathbf{1}_{\tf \geq \frac{3T_0}{4}} \leq \frac{1}{\varphi_3(\frac{3T_0}{4})} \varphi_3(\tf)$, 
    \Be\label{L1_coerc''}
  \|f(NT_0)\|_{L^1_{x,v}}  \leq  (1-\delta_{\mathfrak{m}, T_0} )   \|f((N-1)T_0)\|_{L^1_{x,v}} 
   + \frac{2\delta_{\mathfrak{m}, T_0} }{\varphi_3(\frac{3T_0}{4})} \|  \varphi_3(\tf)f((N-1)T_0)\|_{L^1_{x,v}} .
 \Ee   
 
 Summing up $(\ref{L1_coerc''})+
  \frac{4 \delta_{\mathfrak{m}, T_0} }{T_0\varphi_3 (\frac{3T_0}{4})}
\big\{
(\ref{energy:phi3})|_{t=NT_0}+(\ref{energy:phi4})
\big\}
,$ and choosing $T_0>0$ as in (\ref{cond:T0}), we deduce (\ref{energyi}) for $i=4$.\unhide\end{proof}
 
 Now we are well equipped to prove Proposition \ref{theorem:1}.
 \begin{proof}[\textbf{Proof of Proposition \ref{theorem:1}}]
Fix $T_0$ in (\ref{cond:T0}) and recall norms of $\vertiii{\cdot}_1$ and $\vertiii{\cdot}_4$ in \eqref{|||i}. From (\ref{energyi}), for $i=1,4$, 
     \Be\label{bound|||}
     \vertiii{f(NT_0)}_i \leq  \vertiii{f((N-1)T_0)}_i  \leq \cdots \leq \vertiii{f(0)}_i, \ \ \text{for all } N \in \mathbb{N}.   
     \Ee 

   \textbf{Step 1.} Since $ { \varphi_1 (\tau )} / {\varphi_4 (\tau)}$ is a decreasing function of $\tau\gg 1$, for $M\gg 1$
\big($M$ will be chosen large enough to satisfy \eqref{1-1/M} and \eqref{def:M}\big) , we have 
   $\varphi_1 (\tf)  = \mathbf{1}_{\tf\geq M}   \varphi_1 (\tf) + \mathbf{1}_{\tf< M}   \varphi_1 (\tf) 
  = \mathbf{1}_{\tf\geq M}\frac{ \varphi_1 (M )}{\varphi_4 (M)}  \varphi_4 (\tf)
   +  \mathbf{1}_{\tf<M} M \varphi_0(\tf).$
From the above bound and (\ref{bound|||}) for $i=4$, for $M\gg1$, $N \in \mathbb{N}$, 
  \Be\label{intp:varphi}
  \begin{split}
  \frac{1}{M}  \| \varphi_1(\tf) f((N-1)T_0) \|_{L^1_{x,v}} 
  & \leq \frac{1}{M}\frac{ \varphi_1 (M )}{\varphi_4 (M)}   \| \varphi_4(\tf) f((N-1)T_0) \|_{L^1_{x,v}}
   +  \| \varphi_0(\tf) f((N-1)T_0) \|_{L^1_{x,v}}\\
  & \leq \frac{1}{M}\frac{ \varphi_1 (M )}{\varphi_4 (M)} 
  \frac{T_0  \varphi_{3}(\frac{3T_0}{4})}{ 4\delta_{\mathfrak{m}, T_0} }
     \vertiii{ f(0) }_4
   +  \| \varphi_0(\tf) f((N-1)T_0) \|_{L^1_{x,v}}.
   \end{split}
  \Ee

  From (\ref{energyi}) and (\ref{intp:varphi}), with $\eqref{energy2}_*:=   \max  \big\{ (1-  \delta_{\mathfrak{m},T_0}  \{
1
 -  \frac{ 4 C(2+ T_0)  }{T_0\varphi_0(\frac{3T_0}{4})} \}
   ) ,
  (  \frac{3}{4}+ \frac{1}{T_0} ) ,
 (1- \frac{1}{M} )
    \big\}$,
     \Be\label{energy2}
   \begin{split} 
    \vertiii{ f(NT_0 )}_1 
   \leq\eqref{energy2}_* \times 
       \vertiii{ f((N-1)T_0 )}_1 +\frac{1}{M}\frac{ \varphi_1 (M )}{\varphi_4 (M)}
        \frac{ \varphi_{3}(\frac{3T_0}{4})  }{ \varphi_0(\frac{3T_0}{4})}  
       %   \frac{4 \delta_{\mathfrak{m}, T_0} }{T_0\varphi_0(\frac{3T_0}{4})} \frac{T_0  \varphi_{3}(\frac{3T_0}{4})}{ 4\delta_{\mathfrak{m}, T_0} }
     \vertiii{ f(0) }_4.
      \end{split}
   \Ee
  
  \hide
   \Be\label{energy2}
   \begin{split}   
    & \| f(NT_0) \|_{L^1_{x,v}}
    +    \frac{ 4\delta_{\mathfrak{m}, T_0} }{ \mathfrak{log} (\frac{3T_0}{4})}
 \| \mathfrak{log} (\tf) f  (NT_0)\|_{L^1_{x,v}}
 %+   \frac{4 \delta_{\mathfrak{m}, T_0} }{T_0\mathfrak{log} (\frac{3T_0}{4})} \| \mathfrak{log}(\tf) f(t) \|_{L^1_{x,v}}
  \\
    &   +    \frac{ 4\delta_{\mathfrak{m}, T_0} }{T_0\mathfrak{log} (\frac{3T_0}{4})}\| \varphi_1(\tf) f(NT_0) \|_{L^1_{x,v}}
 +   \frac{2 \delta_{\mathfrak{m}, T_0} }{T_0\mathfrak{log} (\frac{3T_0}{4})}\int^{NT_0}_{(N-1)T_0} |f|_{L^1_{\gamma_+}}
 \\
  \leq& \  \Big(1-  \delta_{\mathfrak{m},T_0} \Big\{
1
 -  \frac{ 4 C(2+ T_0)  }{T_0\mathfrak{log} (\frac{3T_0}{4})}\Big\}
  \Big) 
    \|f((N-1)T_0)\|_{L^1_{x,v}}  
    \\
   %&+   \frac{4 \delta_{\mathfrak{m}, T_0} }{T_0\mathfrak{log} (\frac{3T_0}{4})}  \|  \mathfrak{log}(\tf)f((N-1)T_0)\|_{L^1_{x,v}}
 % +  \frac{4 \delta_{\mathfrak{m}, T_0} }{T_0\mathfrak{log} (\frac{3T_0}{4})} C \| f((N-1)T_0) \|_{L^1_{x,v}} 
 % \\
 &
    + \Big(  \frac{3}{4}+ \frac{1}{T_0}\Big) \frac{4\delta_{\mathfrak{m}, T_0} }{\mathfrak{log} (\frac{3T_0}{4})} \|  \mathfrak{log}(\tf)f((N-1)T_0)\|_{L^1_{x,v}}  \\
    &
+   \Big(1- \frac{1}{M}\Big)  \frac{4 \delta_{\mathfrak{m}, T_0} }{T_0\mathfrak{log} (\frac{3T_0}{4})}
 \| \varphi_1(\tf) f((N-1)T_0) \|_{L^1_{x,v}}  %+  \frac{ 4\delta_{\mathfrak{m}, T_0} }{T_0\mathfrak{log} (\frac{3T_0}{4})}C(1+ T_0)\| f((N-1)T_0) \|_{L^1_{x,v}} 
 \\
 &+ \frac{1}{M}\frac{ \varphi_1 (M )}{\varphi_4 (M)}   \frac{4 \delta_{\mathfrak{m}, T_0} }{T_0\mathfrak{log} (\frac{3T_0}{4})} \frac{T_0  \varphi_{i-1}(\frac{3T_0}{4})}{ 4\delta_{\mathfrak{m}, T_0} }
     \vertiii{ f(0) }_4.
   \end{split}
   \Ee
   \unhide

\textbf{Step 2.} Tentatively we make an assumption, which will be justified later behind (\ref{def:M}), 
   \Be\label{1-1/M}
     \Big(1+\frac{1}{M}\Big)^{-1}\geq 
  %\Big(1- \frac{1}{10}\Big)\geq
   \max\Big\{
  \Big(1-  \delta_{\mathfrak{m},T_0} \Big\{
1
 -  \frac{ 4 C(2+ T_0)  }{T_0 \varphi_0 (\frac{3T_0}{4})}\Big\}
  \Big) , \Big(\frac{3}{4} + \frac{1}{T_0}\Big),    \Big(1- \frac{1}{M}\Big)
  \Big\}.
   \Ee
For $t\geq 0$, choose $N_* \in \mathbb{N}$ such that $t \in [N_* T_0, (N_* +1)T_0]$. From (\ref{energy2}) and (\ref{1-1/M}), we derive, for all $0 \leq N\leq N_*+1$,
   \Be\label{est|||}
   \vertiii{f(NT_0)}_1\leq  \Big(1 +\frac{1}{M}\Big)^{-1}  \vertiii{f((N-1)T_0)}_1
   +  \mathfrak{R}, \ \ \text{with} \    \mathfrak{R}:= \frac{1}{M}\frac{ \varphi_1 (M )}{\varphi_4 (M)}
        \frac{ \varphi_{3}(\frac{3T_0}{4})  }{ \varphi_0(\frac{3T_0}{4})}   
     \vertiii{ f(0) }_4.
   \Ee 
From \eqref{energy_varphi} and $N_* T_0 \leq t$, then there exists $C>0$ such that,
\Be
\begin{split} \label{first term in est:|||1}
  \| \varphi(\tf) f(t) \|_{L^1_{x,v}}
  \leq     \| \varphi(\tf) f(N_* T_0) \|_{L^1_{x,v}}+
 C \| f(N_* T_0) \|_{L^1_{x,v}}.
\end{split}
\Ee 
Now applying \eqref{first term in est:|||1} first and using \eqref{est|||} successively, we conclude that
\begin{align}
 \vertiii{f(t)}_1 
 &  \lesssim \vertiii{f(N_* T_0)}_1
  \leq  \Big(1+ \frac{1}{M}\Big)^{-1}\vertiii{f((N_*-1)T_0)}_1
 + \mathfrak{R}\notag \\
  & \leq  \Big(1+ \frac{1}{M}\Big)^{-2}\vertiii{f((N_*-2)T_0)}_1
+  \Big(1+ \frac{1}{M}\Big)^{-1} \mathfrak{R}
 +  \mathfrak{R} \leq \cdots  \leq  \Big(1 + \frac{1}{M}\Big)^{-N_*}\vertiii{f(0)}_1
+ (1+M) \mathfrak{R} . \label{est:|||1}
\end{align} 
From
$ \big(1 + \frac{1}{M}\big)^{-N_*}   =  \big( (1 + \frac{1}{M} )^{-M}\big)^{ \frac{N_*}{M} } 
 \leq e^{- \frac{N_*}{2M}}
 \leq e^{- \frac{t}{2T_0M}}$,
$(1+M) \mathfrak{R}   \leq  2\frac{ \varphi_1 (M )}{\varphi_4 (M)}  \frac{ \varphi_{3}(\frac{3T_0}{4})  }{ \varphi_0(\frac{3T_0}{4})}   
     \vertiii{ f(0) }_4$, 
we have 
\Be\label{est:|||1'}
 \vertiii{f(t)}_1 \leq (\ref{est:|||1}) \lesssim 
\max  \big\{
e^{- \frac{t}{2T_0M}},  {  \varphi_1 (M )} / {\varphi_4 (M)}\big\}
\big\{
\vertiii{ f(0) }_1
+ \vertiii{ f(0) }_4
\big\}.
\Ee
Following an optimization trick (making $|e^{- \frac{t}{2T_0M}}-{ \varphi_1 (M )} / {\varphi_4 (M)}|\ll1$ as much as possible), choosing  
 \Be\label{def:M}
   M= t [{2T_0 \ln(10+t)^3}]^{-1}, \text{ so that } \max  \big\{
e^{- \frac{t}{2T_0M}},  {  \varphi_1 (M )} / {\varphi_4 (M)}\big\} 
     \lesssim (\ln\langle t\rangle )^{2- \frac{\delta}{2}} \langle t\rangle^{-3}.
   \Ee
Clearly such a choice assures our precondition (\ref{1-1/M}) for $t\gg1$. On the other hand it is straightforward to check $\vertiii{ f(0) }_1
+  \| \varphi_3(\tf) f_0 \|_{L^1_{x,v}}
\lesssim \| e^{\theta^\prime |v|^2} f_0 \|_{L^\infty_{x,v}}$ from \eqref{COV} and \eqref{jacob:mapV}, while $\| \varphi_4(\tf) f_0 \|_{L^1_{x,v}} <\infty$ has been taken for grated from the postulation of Theorem \ref{theorem}.   
Setting  $\varphi=\varphi_1, t_*=t/2$ and
from \eqref{phi'}, we have
$$\frac{t}{2} \| f(t) \|_{L^1_{x,v}} 
\lesssim \int^t_{\frac{t}{2}} \| \varphi_1^\prime (\tf) f(s) \|_{L^1_{x,v}} \dd s 
$$
Applying \eqref{energy_varphi} and \eqref{maximum}, we get
\Be \notag
\int^t_{\frac{t}{2}} \| \varphi_1^\prime (\tf) f(s) \|_{L^1_{x,v}} \dd s 
\lesssim  \vertiii{  f \Big( \frac{t}{2} \Big) }_1
\Ee
From \eqref{est:|||1'} and \eqref{def:M}, we derive
\Be \notag 
\vertiii{  f \Big( \frac{t}{2} \Big) }_1
\lesssim (\ln\langle t\rangle )^{2- \frac{\delta}{2}} \langle t\rangle^{-3}
 \{ \| e^{\theta^\prime |v|^2} f_0 \|_{L^\infty_{x,v}}+ \| \varphi_4(\tf) f_0 \|_{ {L^1_{x,v}}}
 \}.
\Ee
Therefore, we finally prove \eqref{theorem_1}.
\end{proof}

\section{$L^\infty$-Estimates of Moments}
We give proofs for Lemma \ref{lem:bound1} and Lemma \ref{lem:small_largek}.
\hide
\begin{equation}
\begin{split} \label{expand_h}
    & h (t, x, v) 
    = \mathbf{1}_{t_1 \leq t_{*}}   h (t_{*}, x - (t - t_{*}) v, v)
    \\& + \sum\limits^{k}_{i=2} \mathbf{1}_{t_{i-1} \geq t_{*} > t_i}   w (v) \mu(v) \int_{\mathcal{V}_1} \cdots \int_{\mathcal{V}_{i-1}} \frac{h (t_{*}, x_{i-1} - (t_{i-1} - t_{*}) v_{i-1}, v_{i-1})}{w (v_{i-1})} d \sigma_1 \cdots \frac{d \sigma_{i-1}}{\mu(v_{i-1})}
    \\& + \mathbf{1}_{t_k \geq t_{*} > 0}   w (v) \mu(v) \int_{\mathcal{V}_1} d \sigma_1 \cdots \int_{\mathcal{V}_{k-1}} \frac{h (t_k, x_k, v_{k-1})}{w (v_{k-1})} \frac{d \sigma_{k-1}}{\mu(v_{k-1})}.
\end{split}
\end{equation}
\unhide
\begin{proof}[\textbf{Proof of Lemma \ref{lem:bound1}}]
For (\ref{bound1:expand_h}) it suffices to prove this upper bound for 
\Be\label{forcing}
\int_{\mathcal{V}_1} \cdots \int_{\mathcal{V}_{i-1}} 
\int^{t_i}_0 
\int_{\mathcal{V}_i} 
\mathbf{1}_{t_{i+1} < 0 \leq t_i}
\varrho^\prime(s)
|f (s, x_{i} - (t_{i} - s) v_{i}, v_{i})|  \{n(x_i) \cdot v_i\} \dd v_i \dd s  \dd \sigma_{i-1}  \cdots \dd \sigma_1.
\Ee  

\smallskip 
\textbf{Step 1.} Applying Lemma \ref{lem:mapV}, (\ref{mapV}), (\ref{jacob:mapV}) with $x=x_{j}$ and $v=v_{j}$, we derive the change of variables, for $j=i-1, i-2$,  
\Be\notag
v_{j} \in \mathcal{V}_{j} 
\mapsto
(x_{j+1}, t_{\mathbf{b}, j})
:=
(x_{\mathbf{b}} (x_{j}, v_{j}), t_{\mathbf{b}} (x_{j}, v_{j})) \in \partial \Omega \times 
[0, t_{j}], 
 \ \ \text{with} \ \dd v_{j} =  {|t_{\mathbf{b},j}| ^{-4}}  {|n (x_{j+1}) \cdot (x_{j} - x_{j+1}) |}\dd t_{\mathbf{b},j} \dd S_{x_i}.
\Ee Applying above change of variables twice, we derive that (\ref{forcing}) equals 
\begin{align}
 & \int_{\mathcal{V}_1} \dd \sigma_1 \cdots \int_{\mathcal{V}_{i-3}} \dd \sigma_{i-3}\notag \\
 & \times 
 \int_0^{t_{i-2}} \dd t_{\mathbf{b}, i-2}  \int_{\p\O}  \dd S_{x_{i-1}}
c_\mu \mu \Big( \frac{|x_{i-2}- x_{i-1}|}{|t_{\mathbf{b}, i-2}|}\Big)
\frac{|n(x_{i-1}) \cdot (x_{i-2}- x_{i-1})||n(x_{i-2}) \cdot (x_{i-2}- x_{i-1})|}{|t_{\mathbf{b}, i-2}|^5}\notag\\
& \times 
\int_0^{t_{i-2}- t_{\mathbf{b}, i-2}} \dd t_{\mathbf{b}, i-1}  \int_{\p\O}  \dd S_{x_i}
c_\mu \mu \Big( \frac{|x_{i-1}- x_i|}{|t_{\mathbf{b}, i-1}|}\Big)
\frac{|n(x_{i}) \cdot (x_{i-1}- x_i)||n(x_{i-1}) \cdot (x_{i-1}- x_i)|}{|t_{\mathbf{b}, i-1}|^5}  %\frac{|n (x_{i-1}) \cdot (x_{i-1} - x_i) |}{|t_{\mathbf{b},i-1}| ^4} 
\notag\\
&    \times \bigg(
\int^{t_{i-1}  - t_{\mathbf{b}, i-1}}_0 
\int_{\mathcal{V}_i} 
\mathbf{1}_{t_{i-1} - t_{\mathbf{b}, i-1} -\tb(x_i,v_i) < 0  }
\varrho^\prime(s)
|f (s, x_{i}- (t_{i-2} - t_{\mathbf{b},i-2} - t_{\mathbf{b},i-1} - s) v_{i}, v_{i})|  \{n(x_i) \cdot v_i\} \dd v_i \dd s  \bigg),\label{est1:forcing}
\end{align} 
with $t_{i-2}%=t_{i-2} (t,x,v,v_1,\cdots, v_{i-3})
$, $x_{i-2}%=x_{i-2} (t,x,v,v_1,\cdots, v_{i-3})
$ defined in \eqref{def:t_k}, and $t_{i-1}= t_{i-2} - t_{\mathbf{b}, i-2}$. Using (\ref{convex:2}), we bound the above integration as  

\Be\label{est2:forcing}
\begin{split}
&\int_{\mathcal{V}_1} \dd \sigma_1 \cdots \int_{\mathcal{V}_{i-3}} \dd \sigma_{i-3}
\int^{t_{i-2}}_0 \dd t_{\mathbf{b}, i-1}
 \int_0^{t_{i-2}- t_{\mathbf{b}, i-1}} \dd t_{\mathbf{b}, i-2} \int_{\p\O} \dd S_{x_i}
\\
&\times
\underbrace{\bigg(\int_{\p\O} 
 c_\mu \mu \Big( \frac{|x_{i-2}- x_{i-1}|}{|t_{\mathbf{b}, i-2}|}\Big)
\frac{| x_{i-2}- x_{i-1}|^4 }{|t_{\mathbf{b}, i-2}|^5}
c_\mu \mu \Big( \frac{|x_{i-1}- x_i|}{|t_{\mathbf{b}, i-1}|}\Big)
\frac{| x_{i-1}- x_i|^4 }{|t_{\mathbf{b}, i-1}|^5}   \dd S_{x_{i-1}} \bigg)}_{(\ref{est2:forcing})_*}
(\ref{est1:forcing}).
\end{split}\Ee 

\textbf{Step 2.} We claim that 
\Be\label{est3:forcing}
(\ref{est2:forcing})_* \lesssim \mathbf{1}_{t_{\mathbf{b}, i-1} \leq t_{\mathbf{b},i-2}}
\langle t_{\mathbf{b},i-2}\rangle^{-5}
%\Big\{
 %\mathbf{1}_{t_{\mathbf{b},i-2} \leq 1} 
%+ \mathbf{1}_{t_{\mathbf{b},i-2} \geq 1}
%\frac{1}{|t_{\mathbf{b}, i-2}|^5}\Big\}
+  \mathbf{1}_{t_{\mathbf{b}, i-1} \geq t_{\mathbf{b},i-2}}
\langle t_{\mathbf{b},i-1}\rangle^{-5}.
\Ee
We split the cases: \textit{Case 1: $t_{\mathbf{b}, i-1} \leq t_{\mathbf{b},i-2}$. }  Using $|x_{i-2} - x_{i-1}| \lesssim_\O 1$, we bound 
\begin{align}
c_\mu \mu \Big( \frac{|x_{i-2}- x_{i-1}|}{|t_{\mathbf{b}, i-2}|}\Big)
\frac{| x_{i-2}- x_{i-1}|^4 }{|t_{\mathbf{b}, i-2}|^5}&\lesssim  
\mathbf{1}_{t_{\mathbf{b},i-2} \leq 1} \frac{1}{
|t_{\mathbf{b}, i-2}| 
}+ \mathbf{1}_{t_{\mathbf{b},i-2} \geq 1}
\frac{1}{|t_{\mathbf{b}, i-2}|^5} 
,\label{bound1_1and5}\\
 c_\mu \mu \Big( \frac{|x_{i-1}- x_{i}|}{|t_{\mathbf{b}, i-1}|}\Big)
\frac{| x_{i-1}- x_{i}|^4 }{|t_{\mathbf{b}, i-1}|^5}&\lesssim  
 \mu^{\frac{1}{2}} \Big( \frac{|x_{i-1}- x_{i}|}{|t_{\mathbf{b}, i-1}|}\Big)
 \bigg\{
\mathbf{1}_{t_{\mathbf{b},i-1} \leq 1} \frac{1}{
|t_{\mathbf{b}, i-1}| 
}+ \mathbf{1}_{t_{\mathbf{b},i-1} \geq 1}
\frac{1}{|t_{\mathbf{b}, i-1}|^5} \bigg\}.\label{bound2_1and5}
 \end{align} 

We employ a change of variables, for $x_i \in \p\O$ and $t_{\mathbf{b},i-1}\geq0$, $x_{i-1} \in  \p\O
\mapsto z := \frac{1}{t_{\mathbf{b}, i-1}} (x_{i-1}-x_i) \in \mathfrak{S}_{x_i, t_{\mathbf{b}, i-1}},$ where the image $\mathfrak{S}_{x_i, t_{\mathbf{b}, i-1}}$ of the map is a two dimensional smooth hypersurface. Using the local chart of $\p\O$ we have $%\Be\label{jac:COV5}
\dd S_{x_{i-1}} \lesssim |t_{\mathbf{b}, i-1}|^2 \dd S_z.$
%\Ee
From this change of variables and (\ref{bound1_1and5}), (\ref{bound2_1and5}), we conclude that 
\Be\begin{split}\label{bound3_1and5}
 &\mathbf{1}_{t_{\mathbf{b}, i-1} \leq t_{\mathbf{b},i-2}}(\ref{est2:forcing})_*  \lesssim   \mathbf{1}_{t_{\mathbf{b}, i-1} \leq t_{\mathbf{b},i-2}}(\ref{bound1_1and5}) 
\int_{ \mathfrak{S}_{x_i, t_{\mathbf{b}, i-1}} } \mu^{\frac{1}{2}}(z)
\Big\{\mathbf{1}_{t_{\mathbf{b},i-1} \leq 1}|t_{\mathbf{b}, i-1}|+ \mathbf{1}_{t_{\mathbf{b},i-1} \geq 1}
\frac{1}{|t_{\mathbf{b}, i-1}|^3}\Big\}
 \dd S_{z} \\
 &\lesssim  \mathbf{1}_{t_{\mathbf{b}, i-1} \leq t_{\mathbf{b},i-2}}
 \Big\{
 \mathbf{1}_{t_{\mathbf{b},i-2} \leq 1} \frac{1}{
|t_{\mathbf{b}, i-2}| 
}+ \mathbf{1}_{t_{\mathbf{b},i-2} \geq 1}
\frac{1}{|t_{\mathbf{b}, i-2}|^5} \Big\}\Big\{\mathbf{1}_{t_{\mathbf{b},i-1} \leq 1}|t_{\mathbf{b}, i-1}|+ \mathbf{1}_{t_{\mathbf{b},i-1} \geq 1}
\frac{1}{|t_{\mathbf{b}, i-1}|^3}\Big\}\\
&\lesssim
 \mathbf{1}_{t_{\mathbf{b}, i-1} \leq t_{\mathbf{b},i-2}}\Big\{
 \mathbf{1}_{t_{\mathbf{b},i-2} \leq 1} \frac{t_{\mathbf{b},i-1}}{t_{\mathbf{b},i-2}}
+ \mathbf{1}_{t_{\mathbf{b},i-2} \geq 1}
\frac{1}{|t_{\mathbf{b}, i-2}|^5}
 \Big\} \lesssim 
 \mathbf{1}_{t_{\mathbf{b},i-2} \leq 1} 
+ \mathbf{1}_{t_{\mathbf{b},i-2} \geq 1}
\frac{1}{|t_{\mathbf{b}, i-2}|^5}.
 \end{split}\Ee

\textit{Case 2: $t_{\mathbf{b}, i-1} \geq t_{\mathbf{b},i-2}$. } We change the role of $i-1$ and $i-2$ and follow the argument of the previous case.  Using $|x_{i-1} - x_{i}| \lesssim_\O 1$, we bound $
c_\mu \mu \Big( \frac{|x_{i-1}- x_{i}|}{|t_{\mathbf{b}, i-1}|}\Big)
\frac{| x_{i-1}- x_{i}|^4 }{|t_{\mathbf{b}, i-1}|^5}\lesssim  
\mathbf{1}_{t_{\mathbf{b},i-1} \leq 1} \frac{1}{
|t_{\mathbf{b}, i-1}| 
}+ \mathbf{1}_{t_{\mathbf{b},i-1} \geq 1}
\frac{1}{|t_{\mathbf{b}, i-1}|^5} 
,$
 and
$
 c_\mu \mu \Big( \frac{|x_{i-2}- x_{i-1}|}{|t_{\mathbf{b}, i-2}|}\Big)
\frac{| x_{i-2}- x_{i-1}|^4 }{|t_{\mathbf{b}, i-2}|^5}$ $\lesssim  
 \mu^{\frac{1}{2}} \Big( \frac{|x_{i-2}- x_{i-1}|}{|t_{\mathbf{b}, i-2}|}\Big)
 \big\{
\mathbf{1}_{t_{\mathbf{b},i-2} \leq 1} \frac{1}{
|t_{\mathbf{b}, i-2}| 
}+ \mathbf{1}_{t_{\mathbf{b},i-2} \geq 1}
\frac{1}{|t_{\mathbf{b}, i-2}|^5} \big\}.$ We employ a change of variables, for $x_{i-1} \in \p\O$ and $t_{\mathbf{b},i-2}\geq0$, $
x_{i-2} \in  \p\O
\mapsto z := \frac{1}{t_{\mathbf{b}, i-2}} (x_{i-2}-x_{i-1}) \in \mathfrak{S}_{x_2, t_{\mathbf{b}, i-2}} ,
$ with $\dd S_{x_{i-2}} \lesssim |t_{\mathbf{b}, i-2}|^2 \dd S_z.$ Then we can conclude that 
%\Be\begin{split}\label{bound4_1and5}
 $\mathbf{1}_{t_{\mathbf{b}, i-1} \geq t_{\mathbf{b},i-2}}(\ref{est2:forcing})_*  \lesssim 
 \mathbf{1}_{t_{\mathbf{b},i-1} \leq 1} 
+ \mathbf{1}_{t_{\mathbf{b},i-1} \geq 1}
 {|t_{\mathbf{b}, i-1}|^{-5}}.$
 %\end{split}\Ee
 Clearly this bound and (\ref{bound3_1and5}) imply (\ref{est3:forcing}).

 \smallskip
 
 \textbf{Step 3.} Now we use (\ref{est3:forcing}) to (\ref{est2:forcing}). Then we have 
 \begin{align}
 (\ref{forcing}) \lesssim&\int_{\mathcal{V}_1} \dd \sigma_1 \cdots \int_{\mathcal{V}_{i-3}} \dd \sigma_{i-3}
\int^{t_{i-2}}_0 \dd t_{\mathbf{b}, i-1}\langle t_{\mathbf{b}, i-1} \rangle^{-5} 
 \int_0^{\min\{t_{i-2}- t_{\mathbf{b}, i-1},  t_{\mathbf{b}, i-1}  \}} \dd t_{\mathbf{b}, i-2} \int_{\p\O} \dd S_{x_i} (\ref{est1:forcing})
 \label{est_a:forcing}
 \\
 &+ \int_{\mathcal{V}_1} \dd \sigma_1 \cdots \int_{\mathcal{V}_{i-3}} \dd \sigma_{i-3}
\int^{t_{i-2}}_0 \dd t_{\mathbf{b}, i-2} \langle  t_{\mathbf{b}, i-2}\rangle^{-5}
 \int_0^{ \max\{t_{i-2}- t_{\mathbf{b}, i-2},t_{\mathbf{b}, i-2}  \}} \dd t_{\mathbf{b}, i-1} \int_{\p\O} \dd S_{x_i} (\ref{est1:forcing}). \label{est_b:forcing}
\end{align}  

We first consider (\ref{est_a:forcing}). We employ the following change of variables 
$
(x_i, t_{\mathbf{b}, i-2}  ) 
\mapsto y= x_i - (t_{i-2} -t_{\mathbf{b}, i-2} - t_{\mathbf{b}, i-1}-s) v_i
 \in \Omega,
$ with $
|n(x_i) \cdot v_i| \dd S_{x_i} \dd t_{\mathbf{b}, i-2}   =  \dd y$. Applying this change of variables we derive that 
\Be
\begin{split}
\eqref{est_a:forcing}
&\leq\int_{\mathcal{V}_1} \dd \sigma_1 \cdots \int_{\mathcal{V}_{i-3}} \dd \sigma_{i-3}
\int^{t_{i-2}}_0 \dd t_{\mathbf{b}, i-1}\langle t_{\mathbf{b}, i-1} \rangle^{-5}  \int^{t}_0
 \iint_{\O \times\R^3 } 
  \varrho^\prime(s) |f(s,y,v_i)|  \dd v_i\dd y \dd s \lesssim \int^t_0   \|\varrho^\prime f(s)\|_{L^1_{x,v}} \dd s.\notag
\end{split}
\Ee
A bound of (\ref{est_b:forcing}) can be derived exactly as the one for \eqref{est_a:forcing}, using the change of variables $(x_i, t_{\mathbf{b}, i-1}  ) 
\mapsto y= x_i - (t_{i-2} -t_{\mathbf{b}, i-2} - t_{\mathbf{b}, i-1}-s) v_i
 \in \Omega$ with $
|n(x_i) \cdot v_i| \dd S_{x_i} \dd t_{\mathbf{b}, i-1}   =  \dd y$.
\end{proof}

\begin{proof}[\textbf{Proof of Lemma \ref{lem:small_largek}}]\textbf{Step 1.} Define $\mathcal{V}^{\delta}_i := \{ v_i \in \mathcal{V}_i:  {| n(x_{i}) \cdot v_i |} / {|v_i|^2} < \delta \}.$ From (\ref{estimate on delta}), we have $\int_{\mathcal{V}^{\delta}_j} \dd \sigma_j \leq C \delta^2$. On the other hand, from \eqref{convex:2}, we have $
    t_{\mathbf{b}} (x_i, v_i) \geq C_{\Omega} {| n(x_{i}) \cdot v_i |} /{|v_i|^2}$. Therefore if $v_i \in \mathcal{V}_i \backslash \mathcal{V}^{\delta}_i$, we have $t_{\mathbf{b}} (x_i, v_i) \geq C_{\Omega} \delta$.

 If $t_{k}(t,x,v,v_1,\cdots, v_{k-1}) \geq 0$, we conclude such $v_i \in \mathcal{V}_i \backslash \mathcal{V}^{\delta}_i$ can exist at most $[\frac{t  }{C_{\Omega} \delta}] + 1$ times. Denote the combination $\begin{pmatrix}
M \\ N
\end{pmatrix}= \frac{M(M-1) \cdots (M-N+1)}{N(N-1) \cdots 1}= \frac{M!}{N! (M-N)!}$ for $M,N \in \mathbb{N}$ and $M\geq N$. For $0<\delta \ll 1$,  we have
\begin{equation}
\begin{split} \label{sum_bound}
  \int_{\prod_{j=1}^{k-1} \mathcal{V}_j}   \mathbf{1}_{t_{k} (t,x,v,v_1,\cdots, v_{k-1}) \geq 0} \dd \sigma_{k-1} \cdots \dd \sigma_1
  \leq \sum\limits^{[\frac{t  }{C_{\Omega} \delta}] + 1}_{m=0} 
\begin{pmatrix}
k %-1
 \\
m
\end{pmatrix} 
\big( \int_{\mathcal{V}_i^\delta} \dd \sigma_i \big)^{k-m%-1
}
  \leq (C\delta^2)^{k - [\frac{t  }{C_{\Omega} \delta}]}
  \underbrace{ \sum\limits^{[\frac{t  }{C_{\Omega} \delta}] + 1}_{m=0} 
\begin{pmatrix}
k%-1
 \\
m
\end{pmatrix}}_{\eqref{sum_bound}_*}
.
\end{split}
\end{equation}

\textbf{Step 2.} 
Recall the Stirling's formula $\sqrt{2 \pi} k^{k+\frac{1}{2}} e^{-k} \leq k ! \leq k^{k+\frac{1}{2}} e^{-k+1}$ (e.g. \cite{Billingsley}).
%\begin{equation} \label{stirling}
 %  .
%\end{equation}
Using this bound and $(1 + \frac{1}{\mathfrak{a}-1})^{\mathfrak{a}-1} \leq e$, we have, for $\mathfrak{a}\geq 2$
%\begin{equation}
%\begin{split} \label{estimate on n/a}
 $\begin{pmatrix}
k \\
\frac{k}{\mathfrak{a}}
\end{pmatrix} 
  = \frac{k !}{ (k - \frac{k}{\mathfrak{a}}) ! \frac{k}{ \mathfrak{a}} !} 
\leq %frac{e}{2 \pi}
 (\frac{\mathfrak{a}}{\mathfrak{a}-1})^{\frac{\mathfrak{a}}{\mathfrak{a}-1} k} \mathfrak{a}^{\frac{k}{\mathfrak{a}}} \sqrt{\frac{\mathfrak{a}^2}{k (\mathfrak{a}-1)}}  = %\frac{e}{2 \pi} 
 \frac{1}{\sqrt{k}} \bigg( \mathfrak{a}^{\frac{1}{\mathfrak{a}}} \big( \frac{\mathfrak{a}}{\mathfrak{a}-1} \big)^{\frac{\mathfrak{a}}{\mathfrak{a}-1}} \bigg)^k \sqrt{\frac{ \mathfrak{a}^2}{\mathfrak{a}-1}}
  \leq% \frac{e}{2 \pi} 
  \frac{1}{\sqrt{k}} (e \mathfrak{a})^{\frac{k}{\mathfrak{a}}} \sqrt{\frac{\mathfrak{a}^2}{ \mathfrak{a}-1}}.$
%\end{split}
%\end{equation}
Hence, %This ime derive the following bound
\begin{equation}\label{est:sum_com}
\begin{split}
\sum\limits^{[\frac{k}{\mathfrak{a}}]}_{i=1} 
\begin{pmatrix}
k \\
i
\end{pmatrix} \leq 
\frac{k}{\mathfrak{a}} 
\begin{pmatrix}
k \\
\frac{k}{\mathfrak{a}}
\end{pmatrix} \leq \frac{e}{2 \pi} \sqrt{\frac{k}{\mathfrak{a}}} (e \mathfrak{a})^{\frac{k}{\mathfrak{a}}}.
\end{split}
\end{equation}

\smallskip
\textbf{Step 3.} Now we estimate $\eqref{sum_bound}_*$. For fix $0<\delta\ll 1$ which is independent of $t$, choose  
\Be\label{choice:k}
\mathfrak{a} \in \mathbb{N} \ \text{ such that }
(\delta^{2 \mathfrak{a}} e \mathfrak{a})^{\frac{1}{C_\O \delta}} \leq e^{-2}, \ \text{ and set } 
k :=  \mathfrak{a} \Big(\big[\frac{t  }{C_{\Omega} \delta}\big] + 1\Big).
\Ee 
Using (\ref{est:sum_com}), we derive $\eqref{sum_bound}_* \lesssim   \sqrt{\big[\frac{t  }{C_{\Omega} \delta}\big] + 1} \Big( e \frac{k}{[\frac{t  }{C_{\Omega} \delta}] + 1} \Big)^{[\frac{t  }{C_{\Omega} \delta}] + 1} 
\lesssim \sqrt{\big[\frac{t  }{C_{\Omega} \delta}\big] + 1} (e \mathfrak{a})^{[\frac{t  }{C_{\Omega} \delta}] +1 }$ and hence \eqref{sum_bound} is bounded by 
$ (\delta^{2\mathfrak{a}} e \mathfrak{a} )^{[\frac{t  }{C_{\Omega} \delta}] + 1} \sqrt{\big[\frac{t  }{C_{\Omega} \delta}\big] + 1} 
 \lesssim e^{-t}.$ This completes the proof. 
\end{proof}

Equipped with Proposition \ref{theorem:1}, and Lemma \ref{lem:bound1}-\ref{lem:small_largek},
we present a proof of the main theorem:
\begin{proof}[\textbf{Proof of Theorem \ref{theorem}}]
Let $w (v) := e^{\theta |v|^2}, w^\prime (v) := e^{\theta^\prime |v|^2}$ for $0< \theta < \theta^\prime< 1/2.$ It is standard
(\cite{Guo10}) to construct a unique solution of $f$ to (\ref{eqtn_f})-(\ref{diff_f}) and prove its bound $\| w^\prime f(t) \|_{L^\infty_{x,v}}\lesssim \| w^\prime f(0) \|_{L^\infty_{x,v}}$. 
 To utilize the $L^1$-decay of \eqref{theorem_1}, we set 
\Be \label{varrho}
\varrho(t):=  (\ln\langle t\rangle)^{-2}
 \langle t\rangle^{5}.
\Ee
Clearly we have $\varrho^\prime(t)\lesssim  (\ln\langle t\rangle)^{-2}%^{-(2-\e)}
 \langle t\rangle^{4}$ for $t\gg1$.
 
From Lemma \ref{sto_cycle}, we derive the form of $\int_{\R^3} w(v) |f| \dd v$ (i.e. $\varrho=1$). First we split $t_1 \leq 3t/4$ case and get \eqref{first team in f_exp}. Next, for $t_1 \geq 3t/4$ case, we follow along the stochastic cycles twice with $k=2$ and $t_* = t/2$ and get \eqref{second team in f_exp}, \eqref{f_exp2}.
 \begin{align}
 \int_{\R^3} w(v) |f(t,x,v)| \dd v 
 & \leq \int_{\R^3}\mathbf{1}_{t_1 \leq 3t/4}  w(v) |f(3t/4, x- (t-3t/4) v,v)| \dd v
 \label{first team in f_exp} \\
& + \int_{\R^3}  \mathbf{1}_{t_1 \geq 3t/4} c_\mu w(v) \mu(v)  
   \int_{\prod^2_{j=1} \mathcal{V}_j} \mathbf{1}_{t_2 < t/2 < t_1} w (v_1) |f(t/2, x_1 - (t_1 -t/2) v_1, v_1)| \dd \Sigma_{1}^2
  \dd v \label{second team in f_exp} \\
& + \int_{\R^3} \mathbf{1}_{t_1 \geq 3t/4}  
  c_\mu w(v) \mu(v)  
  \Big| \int_{\prod^2_{j=1} \mathcal{V}_j} \mathbf{1}_{t_2 \geq t/2} w (v_2)
f(t_2,x_2,v_2) \dd \Sigma_{2}^2 \Big|
  \dd v.\label{f_exp2}
 \end{align} 
where $\dd  {\Sigma}^{2}_{1} = \dd \sigma_{2}  \frac{ \dd \sigma_{1}}{c_\mu \mu(v_1)w(v_1)} $ and $\dd  {\Sigma}^{2}_{2} = \frac{ \dd \sigma_{2}}{c_\mu \mu(v_2)w(v_2)} \dd \sigma_{1}$,
 with the probability measure $\dd \sigma_j = c_\mu\mu(v_j) \{ n(x_j) \cdot v_j \} \dd v_j$ on $\mathcal{V}_j$ for $j =1, 2$.

For \eqref{first team in f_exp},
considering the change of variables $v \mapsto y=x- (t-3t/4)v \in \O$ where we use $t-\tb(x,v) = t_1 \leq 3t/4$, thus we have $\dd v \lesssim t^{-3} \dd y$. 
Then, from the $L^\infty$-boundedness, $t \gg 1$, $|\O| \lesssim 1$ and $0 < w < w^\prime$,  
we deduce that 
\Be \label{first part in wf}
\eqref{first team in f_exp} 
\lesssim \int_{\O}  w (v) |f(3t/4, y, v)| \langle t\rangle^{-3} \dd y
\lesssim \langle t\rangle^{-3} \| w f(0) \|_{L^\infty_{x,v}} 
\leq \langle t\rangle^{-3} \| w^\prime f(0) \|_{L^\infty_{x,v}}.
\Ee

For \eqref{second team in f_exp}, since $\int_{\R^3} w(v) \mu(v) \dd v \lesssim 1$ and $\dd \sigma_2$ is the probability measure in $\dd {\Sigma}^{2}_{1}$, we have 
\Be \label{first step in second part in wf}
\eqref{second team in f_exp} \lesssim \int_{\mathcal{V}_1} \mathbf{1}_{t_2 < t/2 < t_1} |f(t/2, x_1 - (t_1 -t/2) v_1, v_1)| \{ n(x_1) \cdot v_1 \} \dd v_1
\Ee
Note that $t_1 \geq 3t/4$ implies $ t/4 \leq t_1 -t/2 \leq t$. Considering the change of variables $v_1 \mapsto y =x- (t_1 - t/2)v_1 \in \O$ where we use $t_1 - \tb(x_1,v_1) =t_2 \leq t/2$, clearly we get $\dd v_1 \lesssim t^{-3} \dd y$.
Again, from the $L^\infty$-boundedness, $t \gg 1$, $|\O| \lesssim 1$ and $0 < n(x_1) \cdot v_1 \lesssim w (v_1) < w^\prime (v_1)$,
we derive
\Be \label{second part in wf}
\eqref{first step in second part in wf} 
\lesssim \int_{\O}  w (v_1) |f(t/2, y, v_1)| \langle t \rangle^{-3} \dd y \lesssim \langle t\rangle^{-3} \| w f(0) \|_{L^\infty_{x,v}}
\leq \langle t\rangle^{-3} \| w^\prime f(0) \|_{L^\infty_{x,v}}.
\Ee

Now we only need to bound \eqref{f_exp2}. Since $\int_{\R^3} w(v) \mu(v) \dd v \lesssim 1$ and $\dd \sigma_1$ is the probability measure in $\dd {\Sigma}^{2}_{2}$, it suffices to prove the decay of 
\Be \label{decay of f_exp2}
\sup_{v \in \R^3, v_1 \in \mathcal{V}_1}
\Big| \int_{\mathcal{V}_2} 
\mathbf{1}_{t_2 \geq t/2} f(t_2,x_2,v_2) \{ n(x_{2}) \cdot v_{2} \} \dd v_{2} \Big|. 
\Ee 
Now we define $g := \rho (t_2) w (v_2) f (t_2,x_2,v_2)$, and note that
\Be \notag
\frac{1}{\rho (t_2)} \int_{\mathcal{V}_2} \frac{|n(x_{2}) \cdot v_{2}|}{w (v_2)} g (t_2, x_2, v_2) \dd v_2 = \int_{\mathcal{V}_2} f(t_2,x_2,v_2) \{ n(x_{2}) \cdot v_{2} \} \dd v_{2}.
\Ee 
Therefore, it suffices to show the decay of $\big| \frac{1}{\rho (t_2)} \int_{\mathcal{V}_2} \mathbf{1}_{t_2 \geq t/2} \frac{|n(x_{2}) \cdot v_{2}|}{w (v_2)} g (t_2, x_2, v_2) \dd v_2 \big|$.

Applying Lemma \ref{sto_cycle} with $w(v) = e^{\theta |v|^2}$, $\varrho(t)$ in \eqref{varrho}, and choosing $t_*=0$, $k \geq \mathfrak{C}t$ with $k \sim t$ as in Lemma \ref{lem:small_largek}, we obtain the corresponding expansion of $g (t_2,x_2,v_2) = \rho (t_2) w (v_2) f (t_2,x_2,v_2)$ as \eqref{expand_g1}-\eqref{expand_g3}: for $3 \leq i \leq k$,
\begin{align}
    g (t_2, x_2, v_2) 
   & = \mathbf{1}_{t_3 \leq 0} \
    g (0, x_2 - t_2 v_2, v_2)
\label{expand_g1}
    \\& + \int^{t_2}_{\max(0, t_{3})} \varrho^\prime(s) w(v_2) f(s, x_2 -(t_2-s) v_2, v_2) \dd s
\label{expand_g21}
    \\& +  c_\mu w \mu (v_2)  \int_{\prod_{j=3}^{k} \mathcal{V}_j}   
     \sum\limits^{k-1}_{i=3} 
     \Big\{ \mathbf{1}_{t_{i+1} < 0 \leq t_{i}}  g (0, x_{i} - t_{i} v_{i}, v_{i})  \Big\}
      \dd \tilde{\Sigma}^{k}_{i}
\label{expand_g22}
    \\& + c_\mu w \mu (v_2)  \int_{\prod_{j=3}^{k} \mathcal{V}_j}   
     \sum\limits^{k-1}_{i=3} 
     \Big\{ \mathbf{1}_{0 \leq t_{i}}
     \int^{t_{i}}_{ \max(0, t_{i+1})}w (v_{i}) \varrho^\prime(s) f(s, x_{i} -(t_{i}-s) v_{i}, v_{i}) \dd s
         \Big\}
      \dd \tilde{\Sigma}^{k}_{i} 
\label{expand_g2}
    \\& +  c_\mu w \mu(v_2)  \int_{\prod_{j=3}^{k } \mathcal{V}_j}   
    \mathbf{1}_{t_{k} \geq 0} \
    g (t_{k}, x_{k}, v_{k})
     \dd \tilde{\Sigma}^{k }_{k}, \label{expand_g3} 
\end{align} 
where 
$\dd \tilde{\Sigma}^{k}_{i} :=
\dd \sigma_{k} \cdots \dd \sigma_{i+1} \frac{ \dd \sigma_{i}}{c_\mu \mu(v_{i}) w(v_{i})} \dd \sigma_{i-1} \cdots \dd \sigma_3$. Here, we regard $t_2, x_2, v_2$ as free parameters. 

We will estimate the contribution of \eqref{expand_g1}-\eqref{expand_g3} in 
$\frac{1}{\rho (t_2)} \int_{\mathcal{V}_2} \frac{ n(x_{2}) \cdot v_{2} }{w (v_2)} g (t_2, x_2, v_2) \dd v_2$
term by term. 

For the contribution of \eqref{expand_g1},
we note $t \geq t_2 \geq t/2$ and
consider the change of variables $v_2 \mapsto y = x_2 - t_2 v_2 \in \O$ where we use $t_3 \leq 0$, clearly we have $\dd v_2 \lesssim t^{-3} \dd y$. From the $L^\infty$-boundedness, $t_2 \geq t/2 \gg 1$, $|\O| \lesssim 1$ and $0 < n(x_{2}) \cdot v_{2} \lesssim w (v_2) < w^\prime (v_2)$,  
we deduce that 
\Be \label{est:expand_h1}
\begin{split}
\frac{1}{\rho (t_2)} \int_{\mathcal{V}_2} \frac{|n(x_{2}) \cdot v_{2}|}{w (v_2)} |\eqref{expand_g1}| \dd v_2 
& \lesssim \frac{1}{\rho (t_2)} \int_{\O} \varrho(0) w(v_2) |f (0, y, v_2)| \langle t\rangle^{-3} \dd y  
\\& \lesssim \frac{1}{\rho (t)} \langle t\rangle^{-3} \varrho(0) \| w f(0) \|_{L^\infty_{x,v}}
\lesssim  \frac{1}{\rho (t)} \langle t\rangle^{-3} \| w^\prime f(0) \|_{L^\infty_{x,v}}.
\end{split}
\Ee

Now we bound the contribution of \eqref{expand_g21}. Recall Lemma \ref{lem:bound1} and Proposition \ref{theorem:1} with $\varrho^\prime(t)\lesssim  (\ln\langle t\rangle)^{-2}%^{-(2-\e)}
 \langle t\rangle^{4}$, we have 
\Be \label{est2:expand_h1}
\begin{split}
\frac{1}{\rho (t_2)} \int_{\mathcal{V}_2} \frac{|n(x_{2}) \cdot v_{2}|}{w (v_2)} |\eqref{expand_g21} | \dd v_2  
& \lesssim \frac{1}{\rho (t)} \int^t_0 \| \rho^\prime (s) f(s) \|_{L^1_{x,v}} \dd s    
\lesssim \frac{1}{\rho (t)} \int^t_0 \|  (\ln\langle s\rangle)^{-2}
 \langle s \rangle^{4} f(s) \|_{L^1_{x,v}} \dd s
\\& \lesssim \frac{t}{\rho (t)} \times \{ \| w^\prime f (0) \|_{L^\infty_{x,v}}+ \| \varphi_4(\tf) f (0) \|_{L^1_{x,v}}
 \}. 
\end{split}\Ee

Next, we bound the contribution of \eqref{expand_g22}. From the $L^\infty$-boundedness and
$0 < n(x_{2}) \cdot v_{2} \lesssim w (v_2) < w^\prime (v_2) < \mu^{-1} (v_2)$, we derive 
\Be\label{est1:expand_h1}
\begin{split}
\frac{1}{\rho (t_2)} \int_{\R^3} \frac{|n(x_{2}) \cdot v_{2}|}{w (v_2)} |\eqref{expand_g22}| \dd v_2
& \lesssim \frac{k}{\rho (t)}  \bigg( \sup_{i }    \int_{\prod_{j=1}^{k} \mathcal{V}_j}   
     \mathbf{1}_{t_{i+1}<0 \leq t_{i }} 
      \dd \tilde{\Sigma}^{k}_{i}\bigg)
      \varrho(0)
   \| w f(0) \|_{L^\infty_{x,v}}   
\\&  \lesssim  \frac{k}{\rho (t)}  \bigg(   \int_{n(x_{j}) \cdot v_{j} >0}
        \frac{ |n(x_{j}) \cdot v_{j}| }{w (v_{j})}\dd v_j \bigg)   \| w f(0) \|_{L^\infty_{x,v}}
\lesssim  \frac{k}{\rho (t)}  \| w^\prime f(0) \|_{L^\infty_{x,v}}.
\end{split}
\Ee

Again recall Lemma \ref{lem:bound1} and Proposition \ref{theorem:1}, we bound the contribution of \eqref{expand_g2}. From  $0 < n(x_{2}) \cdot v_{2} \lesssim \mu^{-1} (v_2)$ and $\varrho^\prime(t)\lesssim  (\ln\langle t\rangle)^{-2}%^{-(2-\e)}
 \langle t\rangle^{4}$, we have
\Be\label{est2:expand_h1}
\begin{split}
&\frac{1}{\rho (t_2)} \int_{\R^3} \frac{|n(x_{2}) \cdot v_{2}|}{w (v_2)} |\eqref{expand_g2}| \dd v_2   \\
& \lesssim \frac{k}{\rho (t)} \times  \sup_{i }  \int_{\prod_{j=3}^{k} \mathcal{V}_j}  \mathbf{1}_{0 \leq t_{i}}
\int^{t_{i}}_{ \max(0, t_{i+1})}w (v_{i}) \varrho^\prime(s) f(s, x_{i} -(t_{i}-s) v_{i}, v_{i}) \dd s \dd \tilde{\Sigma}^{k}_{i}    
\\& \lesssim \frac{1}{\rho (t)} \int^t_0 \| \rho^\prime (s) f(s) \|_{L^1_{x,v}} \dd s     \lesssim \frac{k}{\rho (t)} \int^t_0 \|  (\ln\langle s\rangle)^{-2}
 \langle s \rangle^{4} f(s) \|_{L^1_{x,v}} \dd s
\\& \lesssim \frac{k t}{\rho (t)} \times \{ \| w^\prime f (0) \|_{L^\infty_{x,v}}+ \| \varphi_4(\tf) f (0) \|_{L^1_{x,v}}
 \}. 
\end{split}\Ee

Lastly we bound the contribution of \eqref{expand_g3}. From Lemma \ref{lem:small_largek} and $0 < n(x_{2}) \cdot v_{2} \lesssim w (v_2) < w^\prime (v_2)$, we get
\Be \label{est3:expand_h1}
\begin{split}
&\frac{1}{\rho (t_2)} \int_{\R^3} \frac{|n(x_{2}) \cdot v_{2}|}{w (v_2)} |\eqref{expand_g3}| \dd v_2\\
& \lesssim \frac{1}{\rho (t)} \sup_{(x,v) \in \bar{\O} \times \R^3}  \Big(\int_{\prod_{j=1}^{k -1} \mathcal{V}_j}   
    \mathbf{1}_{t_{k }(t,x,v,v_1,\cdots, v_{k-1}) \geq 0 }
\dd \sigma_1 \cdots \dd \sigma_{k-1}\Big)
\sup_{t_k \geq 0} \| w f(t_k) \|_{L^\infty_{x,v}}
\\&  \lesssim \frac{1}{\rho (t)} e^{-t} \| w f(0) \|_{L^\infty_{x,v}}
 \lesssim e^{-t} \| w^\prime f(0) \|_{L^\infty_{x,v}}.
\end{split}
\Ee

Collecting estimates from \eqref{est:expand_h1}-\eqref{est3:expand_h1} and using $k \sim t$, we derive 
\Be \label{g estimate}
\big| \frac{1}{\rho (t_2)} \int_{\mathcal{V}_2} \mathbf{1}_{t_2 \geq t/2} \frac{|n(x_{2}) \cdot v_{2}|}{w (v_2)} g (t_2, x_2, v_2) \dd v_2 \big|
\leq 
\max\{  \frac{1}{\rho (t)} \langle t\rangle^{-3}, \frac{(k+1) t}{\rho (t)}  , e^{-t}
\}   
 \lesssim \frac{\langle t\rangle^2}{\rho (t)}.
\Ee

For \eqref{f_exp2}, using $\varrho(t)= (\ln\langle t\rangle)^{-2}
 \langle t\rangle^{5}$, $0 < w(v) < \mu^{-1} (v)$ and \eqref{g estimate}, we conclude
\Be\notag
\eqref{f_exp2} \lesssim \langle t\rangle^{-3} (\ln\langle t\rangle)^{2}.
\Ee

The above bound, together with \eqref{first part in wf} and \eqref{second part in wf}, proves \eqref{theorem_infty}.
\end{proof}

\section{Acknowledgements}
The authors thank Bertrand Lods sincerely for inspiring discussions and his suggestion to read papers \cite{Bernou,Lods}. They also thank Yan Guo for his interest. This paper is a part of JJ's thesis. This project is partly supported by National Science Foundation under Grant No. 1900923 and the Wisconsin Alumni Research Foundation.

\end{document}